\documentclass[12pt]{article}
\usepackage[labelsep=endash]{caption}

\usepackage{graphicx}
\usepackage{latexsym,amssymb}
\usepackage{amsthm}
\usepackage{indentfirst}
\usepackage{amsmath}
\usepackage{color}
\usepackage{xcolor}

\newtheorem{theoremalph}{Theorem}

\newtheorem{Theorem}{Theorem}[section]
\newtheorem*{Theorem A}{Theorem A}
\newtheorem*{Theorem A'}{Theorem A'}

\newtheorem{Conj}{Conjecture}
\newtheorem*{Conj*}{Conjecture}

\newtheorem{Definition}[Theorem]{Definition}
\newtheorem{Proposition}[Theorem]{Proposition}
\newtheorem{Lemma}[Theorem]{Lemma}

\newtheorem*{Remark}{Remark}

\newtheorem{Remark-numbered}[Theorem]{Remark}
\newtheorem{Remarks-numbered}[Theorem]{Remarks}

\newtheorem*{Claim}{Claim}
\newtheorem*{Theorem B'}{Theorem B'}
\newtheorem{Claim-numbered}{Claim}
\newtheorem*{Acknowledgements}{Acknowledgements}

 \def\RR{{\mathbb R}}

   \def\cN{{\cal N}} 
    \def\cU{{\cal U}}
    \def\cV{{\cal V}}
    
   \def\cR{{\cal R}} \def\cX{{\cal X}}

\def\dim{\operatorname{dim}}

\begin{document}
\title{On the partial hyperbolicity of robustly transitive sets with singularities}
\author {Xiao Wen and Dawei Yang\footnote{D. Yang is the corresponding author. D. Yang was partially supported by NSFC 11671288, NSFC 11822109, NSFC 11790274. X. Wen was partially supported by NSFC 11671025, NSFC 11571188 and the Fundamental Research Funds for the Central Universities.}}

\maketitle

\begin{abstract}
Homoclinic tangencies and singular hyperbolicity are involved in the Palis conjecture for vector fields.  Typical three dimensional vector fields are well understood by recent works. We study the dynamics of higher dimensional vector fields that are away from homoclinic tangencies. More precisely, we prove that for \emph{any} dimensional vector field that is away from homoclinic tangencies, all singularities contained in its robustly transitive singular set are all hyperbolic and have the same index. Moreover,  the robustly transitive set is {$C^1$-generically }partially hyperbolic if the vector field cannot be accumulated by ones with a homoclinic tangency.

\end{abstract}

\section{Introduction}\label{Sec:introduction}

Palis has proposed a sequence of conjectures for understanding typical dynamics. There are lots of progresses of Palis' projects for diffeomorphisms. It is well known that the suspension of diffeomorphisms gives phenomena of flows. However, Lorenz-like dynamics, which are sometimes called the ``butterfly phenomena'', cannot be presented by the suspension of diffeomorphisms. For vector fields, by Palis' program, one has to consider homoclinic tangencies of a regular hyperbolic periodic orbit, singular hyperbolic vector fields which contain hyperbolic ones, etc. It has been proved recently in \cite{CrY17} for the three dimensional case, any typical vector field, which cannot be accumulated by ones with a homoclinic tangency, is globally singular hyperbolic.

However, for the higher dimensional case, there are very few work studying vector fields away from ones with a homoclinic tangency. Meanwhile, there are a sequence of works studying diffeomorphisms that are away from ones with a homoclinic tangencies. The theory of vector fields gives a different picture from the theory of diffeomorphisms because the vector field may contain a non-isolated singularities.

\bigskip

Let $M$ be a
compact $C^\infty$ Riemannian manifold without boundary. Denote by
$\mathcal{X}^1(M)$ the set of all $C^1$ vector fields of $M$ endowed
with the $C^1$ topology. Denote by $\varphi_t=\varphi^X_t$ the flow generated by a vector field $X\in\mathcal{X}^1(M)$. Recall that a point $x\in M$ is a {\it singularity} of $X$ if
$X(x)=0$; and an orbit $Orb(x)$ (under the flow $\varphi_t$) is {\it periodic} (or {\it closed})
if it is diffeomorphic to a circle $S^1$. The singularities and
periodic orbits are the simplest invariant sets of a flow. A singularity or a periodic point is called a \emph{critical point}. But
sometimes, the singularities and the periodic orbits will produce essentially
different dynamics such as in the hyperbolic settings. Let $<X>$ denote the
subspace generated by the vector field $X$. Let $\Phi_t=\Phi_t^X$ be the
tangent flow generated by $\Phi_t= d\varphi^X_T: TM\to TM$. An
invariant set $\Lambda\subset M$ is called {\it hyperbolic} for $X$
if $T_\Lambda M$ has a $\Phi_t$-invariant continuous splitting $E^s\oplus <X>
\oplus E^u$ such that for some constants $C>1,\ \lambda>0$,
$$
\|\Phi_t|_{E_x^s}\|\le Ce^{-\lambda t} {\rm\qquad and\qquad }
\|\Phi_{-t}|_{E_x^u}\|\le Ce^{-\lambda t}
$$
for all $x\in \Lambda$ and $t>0$. The \emph{index} or the stable index of $\Lambda$ is defined to be the dimension of $E^s$ if $\dim E^s$ is constant. Denote by ${\rm Ind}(\Lambda)$ the index of $\Lambda$. It is well known that in a closed
hyperbolic set $\Lambda$, the singularities and periodic orbits
should be separated in two closed subsets. However, there are very famous robust examples as the Lorenz attractor showing that singularities and regular periodic orbits can co-exist in a robustly transitive attractor. This kind of examples leads dynamicists to study compact invariant sets with similar properties in an abstract way.

\bigskip

Denote by $Sing(X)$ the set of all singularities and
$Per(X)$ the set of all periodic points of $\varphi_t^X$. A
singularity $\sigma$ is {\it hyperbolic} if $\{\sigma\}$ is a
hyperbolic set. A periodic orbit $P\subset Per(X)$ is {\it
hyperbolic} if $P$ is a hyperbolic set. 
For a hyperbolic singularity or a hyperbolic periodic point $x$, the {\it index} of $x$ is the index of the hyperbolic set ${\rm Orb}(x)$ and denote it by ${\rm Ind}(x)$. 
Denote by $W^s(P)$ and
$W^u(P)$ the stable and unstable manifolds for a hyperbolic periodic
orbit $P$. We say that $X$ has a {\it homoclinic tangency}
associated to hyperbolic periodic orbit $P$ if there exists a
non-transverse intersection $x\in W^s(P)\cap W^u(P)$, that is,
$Orb(x)\cup P$ is not a hyperbolic set.  This is related to the very famous Newhouse phenomena \cite{New70,New74,New79}.

We would like to emphasize the homoclinic tangencies here is for regular periodic orbits, but not for singularities.

\bigskip

We are interested to characterize dynamics of vector fields that are away from homoclinic tangencies. For diffeomorphisms, higher-dimensional systems away from homoclinic tangencies were studied in lots of works. One of these works is \cite{CSY15} by Crovisier, Sambarino and Yang. They have proved generic diffeomorphisms away from homoclinic tangencies are partially hyperbolic. The extension of \cite{CSY15} to non-singular flows might be regarded to be parallel. In this work, we want to do some extension of \cite{CSY15} to singular flows. However, it seems very difficult to do a full extention of \cite{CSY15} to singular chain-recurrent classes. We put Conjecture~\ref{Con:singular-partially} in Section~\ref{Sec:mutli-partial} in this direction. Now, we will concentrate on a class of sets, which are called robustly transitive sets with singularities. In the three-dimensional case, the Lorenz attractor mentioned above is one prototype of this kind of sets. In higher-dimensional case, we do not know enough about this kind of sets.

In a series of articles, the robust transitivity of $C^1$ vector
fields is studied. In \cite{MPP}, Morales, Pacifico and Pujals
proved that every robustly transitive singular set for a
3-dimensional flow must be partially hyperbolic. A similar results
in higher dimensional manifolds is got in \cite{LGW}, the authors
proved that under a strong homogenuous condition, a robustly
transitive set for a star flow must be partially hyperbolic. Some
subsequent works along this direction can be seen in \cite{MM} and
\cite{ZGW08}.

A compact invariant set
$\Lambda$ of $X$ or $\varphi^X_t$ is  {\it transitive} if there
is $x\in\Lambda$ such that $\Lambda=\omega(x)$. A transitive set
$\Lambda$ is said to be non-trivial if it is neither a singularity nor a
periodic orbit. A compact invariant set $\Lambda$ of $X$ is {\it isolated} or {\it locally maximal} if there is an open neighborhood $U$ of $\Lambda$ such that
$$\Lambda=\bigcap_{t\in\mathbb{R}}\varphi_t^X(U).$$
Here the neighborhood $U$ is called an isolated neighborhood of $\Lambda$.
\begin{Definition}
A compact locally maximal invariant set $\Lambda$ of $\varphi_t^X$ is
\emph{robustly transitive} if there are a neighborhood
$\mathcal{U}\subset \mathcal{X}^1(M)$ of $X$ and a neighborhood
$U\subset M$ of $\Lambda$ such that for any $Y\in\mathcal{U}$,
$$
\Lambda_Y=\bigcap_{t\in \mathbb{R}}\varphi^Y_{t}(U)
$$
is a non-trivial transitive set.
\end{Definition}

In this article, we firstly characterize the singularities in a
robustly transitive set for a vector field which is far away from
homoclinic tangencies.

\begin{theoremalph}\label{Thm:uniqueindex}

If $\Lambda$ is a robustly transitive set of $X\in{\cal X}^1(M)$,
then either all the singularities in $\Lambda$ are hyperbolic and
have the same index, or $X$ can be accumulated by vector fields with
a homoclinic tangency.

\end{theoremalph}

{~Note that the assumption ``robust transitivity'' cannot be relaxed. A recent  work of da Luz \cite{daL18} gave an example: under the star assumption, there is robustly \emph{chain-transitive} chain recurrent classes with singularities of different indices. A way to read Theorem~\ref{Thm:uniqueindex} is: under the assumption ``far away from tangencies'', all singularities in any robustly transitive set have the same index.}

Recall that singularities and periodic orbits in a hyperbolic
invariant set should be separated. Hence the singularities and
hyperbolic periodic orbits can not co-exist in a hyperbolic robustly
transitive set. But in general, hyperbolic singularities and
hyperbolic periodic orbits can co-exist in a robustly transitive set. The famous geometric Lorenz attractor is an example
of this phenomenon. To describe such phenomena, we need to relax the
notions of hyperbolicity.

\begin{Definition} Let $\Lambda$ be a compact invariant set of
$\varphi_t^X$. We say that $\Lambda$ is partially hyperbolic if there is a
$\Phi_t$-invariant splitting $T_\Lambda M=E\oplus F$ with constants
$C\geq 1$ and $\lambda>0$ such that the following properties are
satisfied:
\begin{enumerate}
\item $T_\Lambda M=E\oplus F$ is a $(C,\lambda)$-dominated splitting
for the tangent flow $\Phi_t$, i.e.,
$$\|\Phi_t|_{E_x}\|\cdot\|\Phi_{-t}|_{\Phi_t(F_x)}\|\le Ce^{-\lambda
t}$$ for any $x\in\Lambda$ and $t\geq 0$,
\item there is a dichotomy as following: either $E$ is contracting, that is, $
\|\Phi_t|_{E_x}\|\leq Ce^{-\lambda t} $ for any $x\in\Lambda$ and
$t\geq 0$; or $F$ is expanding, that is, $ \|\Phi_{-t}|_{F_x}\|\leq
Ce^{-\lambda t} $ for any $x\in\Lambda$ and $t\geq 0$.
\end{enumerate}
\end{Definition}

In the article, we give a characterization of robustly transitive
set for $C^1$-generic vector fields far away from homoclinic tangencies. Recall that $\mathcal{R}\subset \mathcal{X}^1(M)$ is called a {\it residual set} if it contains a countable intersection of countably many open and dense subset of $\mathcal{X}^1(M)$. A property for vector fields is said to be \emph{generic} if it is satisfied for a system in a residual set of $\mathcal{X}^1(M)$.

\begin{theoremalph}\label{Thm:partial-hyperbolicity}
There is a residual set $\mathcal{R}\subset \mathcal{X}^1(M)$ such that for any $X\in\mathcal{R}$, if $X$ is far away from homoclinic tangencies,
then every robustly transitive set $\Lambda$ of $X$ is partially
hyperbolic.
\end{theoremalph}

{~ Bonatti and da Luz \cite{BdL17} introduced a new notion ``multi-singular hyperbolicity'' to characterize higher-dimensional star  flows with singularities. The property that is ``far away from homoclinic tangencies'' is generally weaker than the condition ``star''. So generally we can not image very strong hyperbolicity like ``singular hyperbolicity'' or ``multi-singular hyperbolicity'' under the assumption ``far away from homoclinic tangencies''.
}


By the techniques of this paper, one can improve the statement of the main theorem in \cite{ZGW08} a little bit. Recall that a compact invariant set $\Lambda$ of a vector field $X\in\mathcal{X}^1(M)$ is {\it star} if there exist a neighborhood $\mathcal{U}\subset \mathcal{X}^1(M)$ of $X$ and a neighborhood $U\subset M$ of $\Lambda$ such that for any $Y\in\mathcal{U}$ and any periodic orbit $P$ of $Y$ containing in $U$ is hyperbolic. And we say $\Lambda$ is {\it strongly homogenous} if there exist a neighborhood $\mathcal{U}\subset \mathcal{X}^1(M)$ of $X$ and a neighborhood $U\subset M$ of $\Lambda$ such that for any $Y\in\mathcal{U}$, all periodic orbit $P$ of $Y$ contained in $U$ have the same index. One can see that strong homogeneity implies star condition automatically. The converse is true under an additional assumption that $\Lambda$ contains a regular periodic orbit by \cite[Lemma 1.6]{LGW}. Here we can get a following conclusion.
\begin{theoremalph}\label{Thm:improveZGW08}\footnote{We think that Theorem~\ref{Thm:improveZGW08} is folklore now. However, it is worth to clarify in the literature.}
If $\Lambda$ is a robustly transitive set, and $\Lambda$ is star, then every singularity in $\Lambda$ is hyperbolic and $\Lambda$ is strongly homogenous.

\end{theoremalph}
Note that we do not have to assume that any singularity in the robustly transitive set is hyperbolic and we can get the strong homogeneity from the star property.

\paragraph{Further remarks.}
In the long preparation of this work, we noticed a paper by A. da Luz \cite{daL19}. We were also told \cite{daL19} is one part of her thesis. She tried to study robustly chain-transitive sets and use the notion ``singular volume partial hyperbolicity'' to characterize robustly chain-transitive sets. In this paper, we try to study robustly transitive set that are away from homoclinic tangencies, and we use the notion ``partial hyperbolicity'' and ``multi-singular parital hyperbolicity''. The notions in this work and in \cite{daL19} are slightly different. We would say the philosophies are similar. These works are inspired by several papers for diffeomorphisms. We give a partial list of papers from diffeomorphisms: \cite{BDP,BGV,CSY15}. As we mentioned, this work is mainly inspired by \cite{CSY15}. \cite{daL19} may mainly be inspired by \cite{BDP}. 

\cite{ZGW08} studied robustly transitive sets under the ``strongly homogeneous'' condition, which is stronger than ``star'' by definition. We can extend the result of \cite{ZGW08} to the star case because the theory has been developed. So we give Theorem~\ref{Thm:improveZGW08}.

One of the main ingredients for singular star flow is the ``blow-up'' started by the work of Li-Gan-Wen \cite{LGW}. Bonatti-da Luz \cite{BdL17} used this idea to give a very good description of star flows by introducing cocycles. In the proof of this work, we also use lots of extended linear Poincar\'e flow as a tool. The usage may be slightly different from \cite{LGW,BdL17,daL19}.

\section{Preliminary}

\paragraph{Connecting lemma and Ma\~n\'e's ergodic closing lemma for flows.} In the article, we need the following version of $C^1$
connecting lemma from \cite{WX} to perturb the system. As usual, denote by $B(\Lambda,\varepsilon)$ the $\varepsilon$ neighborhood of a closed set $\Lambda$ in $M$.

\begin{Lemma}[\cite{WX}, connecting lemma]\label{Lem:connectinglemma}
Let $X\in\mathcal{X}^1(M)$ and $z\in M$ be neither singularity nor
periodic point. Then for any $C^1$ neighborhood
$\mathcal{U}\subset\mathcal{X}^1(M)$ of $X$, there are constants
$L>1, T>1$ and $\delta_0>0$ such that for any $0<\delta\leq
\delta_0$ and any two points $x,y$ outside the tube
$\Delta=\cup_{t\in[0,T]}B(\varphi_t^X(z),\delta)$, if the positive
$\varphi_t^X$-orbit of $x$ and the negative $\varphi_t^X$ orbit of $y$
both hit $B(z,\delta/L)$, then there exists $Y\in\mathcal{U}$
with $Y=X$ outside $\Delta$ such that $y$ is on the positive
$\varphi_t^Y$-orbit of $x$. Moreover, the resulted $\varphi_t^Y$-orbit
segment from $x$ to $y$ meets $B(z,\delta)$.
\end{Lemma}

Let $\Lambda$ be a compact invariant set of $X$, if there exist vector fields $Y_n\to Y$ in $C^1$ topology and
periodic orbits $P_n$ of $Y_n$ such that $\Lambda$ is the Hausdorff
limit of $P_n$ as $n\to\infty$, then we call $\Lambda$ is a {\it periodic limit}.
As an application of the connecting lemma\footnote{This can be in fact deduced by Pugh's closing lemma \cite{Pug67}.}, we can have the following result, whose proof is omitted.

\begin{Lemma}\label{Lem:transitive-periodic-likt}

Any transitive set is a period limit.

\end{Lemma}

Another important perturbation technique is the ergodic closing lemma first given by Ma\~n\'e in \cite{Mane}. Recall that a point $x\in M\setminus Sing(X)$ is called a {\it well closable
point } of $X$ if for any $C^1$ neighborhood $\mathcal{U}$ of $X$, there is $T>0$ such that
for any $\delta>0$, there exist $Y\in \mathcal{U}$ and a periodic point
$z\in M$ of $Y$ with periodic $\tau$ such that
the following are satisfied:

\begin{enumerate}
\item $d(\varphi^X_t(x),\varphi^Y_t(z))<\delta$ for all $0\leq t\leq
\tau$,
\item $X=Y$ on $M\setminus B_{\delta}(\varphi^X_{[-T,0]}(x))$ where
$B_{\delta}(\varphi^X_{[-T,0]}(x))$ denotes the $\delta$ neighborhood of
the arc $\varphi^X_{[-T,0]}(x)$.
\end{enumerate}
Denoted by $\Sigma(X)$ be the set of the well closable point of $X$. Here we use the following version of ergodic closing lemma first proved in \cite{Wen1}.
\begin{Proposition}[\cite{Wen1}]\label{Pro:ergodicclosinglemma}
For every $X\in\mathcal{X}^1(M)$ and every Borel probability measure $\mu$ which is invariant by the flow $\varphi_t$, one has
$\mu(\Sigma(X))=1$.
\end{Proposition}

\paragraph{Linear Poincar\'e flow and its extension.} We will introduce the linear Poincar\'e flow proposed firstly by Liao and the extended linear Poincar\'e flow firstly by Li-Gan-Wen \cite{LGW}. For any regular point $x\in M\setminus Sing(X)$, denote the normal space of $X$ at $x$ by
$$N_x=<X(x)>^\perp=\{v\in T_x M| v\bot X(x)\}.$$
Denote the {\it normal bundle} of $X$ by
$$N=N(X)=\bigcup_{x\in M\setminus
Sing(X)} N_x.$$
Let $\pi: T_{M\setminus{Sing(X)}}M\to N$ be the canonical projection.

The {\it linear Poincar\'e flow} $\psi_t=\psi_t^X:N\to N$ of $X$  is then defined to be the orthogonal projection of $\Phi_t|_N$ to $N$, i.e.,
$$\psi_t(v)=\pi\circ \Phi_t(v),~~~v\in N_x.$$

One can compactify the linear Poincar\'e flow to be the {\it extended linear Poincar\'e flow}.
Denote by $SM=\{e\in TM:\|e\|=1\}$ the unit sphere bundle of $M$ and $\rho:SM\to M$  the canonical bundle projection.
The tangent flow $\Phi_{t}$ thus induces a flow
$$\Phi_{t}^{\#}:SM \to SM$$
$$\Phi_{t}^{\#}(e)=\Phi_{t}(e)/\|\Phi_{t}(e)\|.$$
For each $e\in SM$, denote by
$$N_e=\{v\in T_{\rho(e)}M: v\perp e\}$$
 the normal space of $e$. Let
$$N_{SM}=\bigcup_{e\in SM}N_e.$$
One can define the following ``linear Poincar\'e flow":
$$\tilde\psi_t: N_{SM}\to N_{SM}$$
by
$$\tilde\psi_t(v)=\pi_{\Phi^\#_t(e)}\circ\Phi_t(v).$$
where $\pi_{\Phi^\#_t(e)}$ is the orthogonal projection from $T_{\rho(\Phi^\#_t(e))}M$ to $N_{\Phi^\#_t(e)}$.


When we take $e=X(x)/\|X(x)\|$ for some regular point $x$, then $N_e=N_x(X)$ and
 $$\tilde\psi_t|_{N_e}=\psi_t|_{N_x}.$$
 By this reason, Li-Gan-Wen \cite{LGW} have defined the compatification of $\psi_t$ on $N_{SM}$ to be the \emph{extended linear Poincar\'e flow}, which is also denoted by $\tilde\psi_t$.

 In other words, the extended linear Poincar\'e flow $\tilde\psi_t$ over the subset $$\{X(x)/\|X(x)\|: x\in M\setminus Sing(X)\}$$ of $SM$ can be identified with the usual linear Poincar\'e flow $\psi_t$ over  $M\setminus Sing(X)$.

Let $\Lambda\subset M$ be a compact invariant set of $X$. Denote by
$$\tilde\Lambda=\overline{\{X(x)/\|X(x)\|: x\in \Lambda\setminus Sing(X)\}}, $$ where the closure is taken in $SM$. The set $\tilde\Lambda$ is compact and $\Phi^{\#}_t$-invariant. Due to the parallel feature of vector fields near regular points, at every $x\in \Lambda\setminus Sing(X)$, $\tilde\Lambda$ gives a single unit vector $X(x)/\|X(x)\|$. Thus in a sense $\tilde\Lambda$ is a ``compactification" of $\Lambda\setminus Sing(X)$.

For a continuous function $h:~\Lambda\setminus{ Sing}(X)\to \RR$, for any sequence $\{x_n\}\subset\Lambda\setminus{Sing}(X)$ converging to $x$, if $X(x_n)/\|X(x_n)\|$ and $\{h(x_n)\}$ are also converging to some limits, then $h$ can be extended to be a continuous function $\tilde h:~\tilde\Lambda\to \RR$.

For measures, we have the following observations:
\begin{itemize}
\item for any measure $\mu$ supported on $\Lambda$, if $\mu(Sing(X))=0$, then one can lift $\mu$ to be $\tilde\mu$ supported on $\tilde\Lambda$ in a natural way. $\tilde\mu$ is said to be the \emph{transgression} of $\mu$.
\item for any invariant ergodic measure $\tilde\mu$ of $\Phi^\#$ supported on $\tilde\Lambda$, if $\tilde\mu(S_{Sing(X)}M)=0$, then $\tilde\mu$ can be projected to be an ergodic measure $\mu$ supported on $\Lambda$ in a unique way with the property that $\mu(Sing(X))=0$.
\end{itemize}
Furthermore, $\mu$ and $\tilde\mu$ have the same Lyapunov exponents.

\bigskip

Note that we have a following easy lemma for transitive sets.
\begin{Lemma}\label{Lem:nonemptyonmanifolds}
Let $\Lambda$ be a non-trivial transitive set of $X$ and $\sigma\in\Lambda$ be a hyperbolic singularity. Then $(W^s(\sigma)\setminus\{\sigma\})\cap \Lambda\neq\emptyset$ and $(W^u(\sigma)\setminus\{\sigma\})\cap \Lambda\neq\emptyset$.
\end{Lemma}
\begin{proof}
By the stable manifold theorem we know that there is $r\geq 0$ such that the local stable manifold of $\sigma$ can be characterized by
$$W^s_r(\sigma)=\{y\in M: d(\varphi^X_t(y),\sigma)\leq r, \forall t\geq 0\}.$$
Since $\Lambda$ is transitive, there is a point $x\in\Lambda$ such that $\Lambda=\omega(x)$. Now we can take a sequence of $x_n$ in the positive orbit of $x$ with $s_n>0$ such that $d(x_n,\sigma)\to 0$ and $d(\varphi_{-s_n}(x_n),\sigma)=r$ and $d(\varphi_{t}(x_n),\sigma)<r$ for all $t\in(-s_n,0)$. By the choice of the sequence $x_n$ and $s_n$, one can see that $s_n\to+\infty$. Let $y$ be an accumulated point of $\varphi_{-s_n}(x_n)$, then we know that $d(y,\sigma)=r$ and $d(\varphi_t(y), \sigma)\leq r$ for all $t\geq 0$. Hence $y\in (W^s(\sigma)\setminus \{\sigma\})\cap \Lambda$. Similarly, we can get that $W^u(\sigma)\setminus\{\sigma\}\cap \Lambda\neq\emptyset$.
\end{proof}

So we have the following lemma for $\tilde\Lambda$.

\begin{Lemma}\label{Lem:nonemptyonspace}
Let $\Lambda$ be a non-trivial transitive set of $X$ and $\sigma\in\Lambda$ be a hyperbolic singularity with hyperbolic splitting $T_\sigma M=E_\sigma^s\oplus E_\sigma^u$. Then we have $\tilde\Lambda\cap E_\sigma^s\neq\emptyset$ and $\tilde\Lambda\cap E_\sigma^u\neq\emptyset$.
\end{Lemma}

\begin{proof}
By Lemma \ref{Lem:nonemptyonmanifolds} we know that there is $x\in\Lambda\cap
W^s(\sigma)\backslash\{\sigma\}$, and then
$\frac{X(x)}{|X(x)|}\in\tilde{\Lambda}$. Take a sequence
$t_n\to+\infty$ such that
$\Phi_{t_n}^\#(X(x)/\|X(x)\|)$ accumulates to some unit
vector $e$, hence $e\in\tilde{\Lambda}\cap
E_\sigma^s$. Similarly we have $\tilde\Lambda\cap E_\sigma^u\neq\emptyset$.
\end{proof}

\bigskip

If $\Lambda$ is a robustly transitive set of $X$ with a neighborhood  $U$ as in definition, then the following extension $B(\Lambda)$ of $\Lambda$ in $SM$ will be more suitable. 

These compactifications started from Li-Gan-Wen \cite{LGW}. For recent results, one can see \cite{BdL17}. There are also compactifications for non-linear dynamics in \cite{GaY18} and \cite{CrY17}.

Here
$$B(\Lambda)=\{e\in SM: \rho(e)\in\Lambda, \exists X_n\to X, \text{ and } p_n\in Per(X_n) $$
$$\ \ \ \ \ \ \ \ \ \ \ \ \ \ \ \ \  Orb_{X_n}(p_n)\subset U, ~\text{s.t.~}
\frac{X_n(p_n)}{\|X_n(p_n)\|}\to e \}.$$
The set $B(\Lambda)$ is also a compact invariant set of $\Phi_t^{\#}$. 
\begin{Lemma}\label{Lem:wildecontainedinB}
If $\Lambda$ is robustly transitive, then $\tilde\Lambda\subset B(\Lambda)$.
\end{Lemma}
\begin{proof}
Fix $e\in\tilde\Lambda$. Given any $\varepsilon>0$, there is a regular point $x\in\Lambda$ such that $X(x)/\|X(x)\|$ is $\varepsilon/2$-close to $e$. By Lemma~\ref{Lem:transitive-periodic-likt}, $\Lambda$ is a periodic limit. So there is a sequence of vector fields $\{X_n\}$ such that each $X_n$ has a periodic point $p_n$ with the property $\lim_{n\to\infty}p_n=x$. By the fact that $x$ is a regular point of $X$, one has that
$$\lim_{n\to\infty}\frac{X(p_n)}{\|X(p_n)\|}=\frac{X(x)}{\|X(x)\|}.$$
This implies that $X(x)/\|X(x)\|$ is contained in $B(\Lambda)$. So $e$ is $\varepsilon/2$-close to $B(\Lambda)$. By the fact that $B(\Lambda)$ is compact, one has that $e\in B(\Lambda)$.
\end{proof}


\paragraph{Dominated splittings.}Similar to the dominated splitting with respect to tangent flow on a $\varphi_t$ invariant set, we can define the dominated splitting with respect to the linear Poincar\'e flow. Let $\Lambda$ be a compact invariant set of $\varphi_t$. We say that a $\psi_t$ invariant splitting $N_{\Lambda\setminus Sing(X)}=N_1\oplus N_2$ is an $l$-{\it dominated splitting} (with respect to $\psi_t$) or bundle $N_1$ is $l$-dominated by bundle $N_2$ if
$$\|\psi_t|_{N_1(x)}\|\cdot \|\psi_{-t}|_{N_2(\varphi_t(x))}\|\leq 1/2$$
for all $x\in\Lambda$ and all $t\geq l$. If  there is an positive integer $i$ such that $\dim N_1(x)=i$ for all $x\in\Lambda\setminus Sing(X)$, then we say the dominated splitting is {\it homogenous} and $i$ is the {\it index} of the dominated splitting. The dominated splitting for the extended linear Poincar\'e flow $\tilde{\psi}_t$ on a $\Phi_t^{\#}$ invariant set can be defined similarly.

We remark here that an equivalent definition for dominated splitting can be given as following. Usually, a $\psi_t$-invariant bundle $N_1$ is dominated by a $\psi_t$-invariant bundle $N_2$ on an invariant set $\Lambda$ means that there exist constants $C\geq 1, \lambda>0$ such that
$$\|\psi_t|_{N_1(x)}\|\cdot \|\psi_{-t}|_{N_2(\varphi_t(x))}\|\leq C{\rm e}^{-\lambda t}$$
for any $x\in\Lambda\setminus Sing(X)$ and $t\geq 0$. We say an invariant splitting $N_1\oplus N_2\oplus\cdots\oplus N_k$ is dominated means that $N_{i}$ is dominated by $N_{i+1}$ for every $i=1,2,\cdots,k-1$. It is known that $N_1\oplus N_2\oplus\cdots\oplus N_k$ is a dominated splitting is equivalent to $(N_1\oplus\cdots N_i)\oplus (N_{i+1}\oplus\cdots\oplus N_k)$ is a dominated splitting for all $i=1,2,\cdots,k-1$.

For the relationship between the dominated splittings for the linear Poincar\'e flow and the extended linear Poincar\'e flow, we have the following proposition.

\begin{Proposition}[\cite{LGW}]\label{Pro:extendedLPF}
If $\Lambda$ admits a dominated splitting
$N_{\Lambda\setminus Sing(X)}=N_1\oplus N_2$ with respect to the
linear Poincar\'e flow, then one has a dominated splitting $N_{\tilde{\Lambda}}=N_1\oplus
N_2$ with respect to the extended linear Poincar\'e
flow $\tilde{\psi}_t^X$ such that $N_1(e)=N_1(\rho(e))$ and $N_2(e)=N_2(\rho(e))$ for all $e\in\tilde\Lambda$ with $\rho(e)\in\Lambda\setminus Sing(X)$.
\end{Proposition}

For more discussion on the extended linear Poincar\'e flow and the dominated splitting, one can see section 2 and section 3 of \cite{LGW}.

\paragraph{Estimation on the periodic orbits.} Here we collect some known results from \cite{Wen2, Wen3} for the systems far away from homoclinic tangencies. Let $\overline{HT}\subset \mathcal{X}^1(M)$ be the closure of the set of vector fields which has a homoclinic  tangency. Let $P$ be a hyperbolic periodic orbit of a $C^1$ vector field $Y$ with period $\pi(P)$. Then for any point $p\in P$, the normal space $N_p$ of $Y$ at $p$ can be split into $N_p^s\oplus N_p^u$ where $N_p^s$ is the sum of eigenspaces related to the eigenvalues of $\psi^Y_{\pi(P)}|_{N_p}$ with modulus less than $1$, $N_p^u$ is the sum of eigenspaces related to the eigenvalues of $\psi^Y_{\pi(P)}|_{N_p}$ with modulus greater than $1$.

\begin{Proposition}[\cite{Wen2}]\label{Pro:dominatedsplitting}
Let $X\in \mathcal{X}^1(M)\setminus\overline{HT}$, then there exist
a $C^1$ neighborhood $\mathcal{U}$ of $X$ and constant
$l>0$ such that for any $Y\in\mathcal{U}$ and any hyperbolic
periodic orbit $P$ of $Y$, the splitting
$N_{P}=N^s(P)\oplus N^u(P)$ is an $l$-dominated
splitting with respect to the linear Poincar\'e flow $\psi_t^{Y}$.
\end{Proposition}

Actually, in \cite{Wen2} it is proved that once $X$ is far away from homoclinic tangencies, then there exist a neighborhood $\mathcal{U}\subset\mathcal{X}^1(M)$ and a constant $a>0$ such that for any $Y\in\mathcal{U}$ and any periodic point $p$ of $Y$ with period $\pi(p)$, there exists at most one eigenvalue of $\psi_{\pi(p)}^Y|_{N_p}$ with modulus in $(e^{-a\pi(p)}, e^{a\pi(p)})$. Then for any periodic orbit $P$ of $Y\in\mathcal{U}$, $N_P$ can be split into $N^{ss}\oplus N^c\oplus N^{uu}$ where three bundles are the sum of eigenspaces associated to eigenvalues with modulus in $(0,e^{-a\pi(P)}]$, $(e^{-a\pi(P)}, e^{a\pi(P)})$ and $[e^{a\pi(P)}, +\infty)$ respectively. The following proposition is proved in \cite{Wen2, Wen3}. Wen stated his theorems for diffeomorphisms in \cite{Wen2,Wen3}. Under the help of the Franks' Lemma for flows \cite[Theorem A.1]{BGV}, one can adapt Wen's result for flows.

\begin{Proposition}[\cite{Wen2, Wen3}]\label{Pro:dominatedsplitting2}
Let $X\in \mathcal{X}^1(M)\setminus\overline{HT}$, then there exist
a $C^1$ neighborhood $\mathcal{U}$ of $X$ and constants $C>1$,
$l>0$ and $\eta>0$ such that for any $Y\in\mathcal{U}$ and any hyperbolic
periodic orbit $P$ of $Y$, the following are satisfied:
\begin{enumerate}
\item $N^c$ has at most dimension one,
\item the splitting $N^{ss}\oplus N^c\oplus N^{uu}$ are $l$-dominated,
\item for any point $p\in P$, we have
$$\prod_{i=0}^{[\pi(P)/l-1]}\|\psi^Y_l|_{N^{ss}(\varphi^Y_{il}(p))}\|<C e^{-\eta\pi(P)},$$
$$\prod_{i=0}^{[\pi(P)/l-1]}\|\psi^Y_{-l}|_{N^{uu}(\varphi^Y_{-il}(p))}\|<C e^{-\eta\pi(P)}.$$

\end{enumerate}
\end{Proposition}




\paragraph{Generic results.} For any two hyperbolic periodic orbits $P_1$ and $P_2$ of a $C^1$ vector field $X$, for any open set $U$, we say $P_1$ is {\it homoclinic related} to $P_2$ in $U$ if the stable manifold of $P_1$ and the unstable manifold of $P_2$ have a transverse intersection point whose orbit is in $U$, and the stable manifold of $P_2$ and the unstable manifold of $P_1$ have a transverse intersection point whose orbit is in $U$. Let $P$ and a neighborhood of $P$ be fixed, we call
$$H(P,U)=\overline{\{x\in P':P'\subset U, P' \text{ is homoclinic related to } P~\text{in}~U\}}$$
the {\it relative homolinic class }of $P$ in $U$. As usual, a $C^1$ vector field $X$ is called a Kupka-Smale system if every critical point of $X$ is hyperbolic, and for any two critical orbit $O_1$ and $O_2$ of $X$, the stable manifold of $O_1$ and the unstable manifold of $O_2$ intersect transversely. Given $\delta>0$, a sequence $\{(x_i,t_i):x_i\in M, t_i\geq 1\}$ is called a $\delta$-{\it pseudo orbit} if $d(\varphi_{t_i}(x_i),x_{i+1})<\delta$ for any $i$. We say that a compact invariant set $\Lambda$ is {\it chain transitive }if for any $x,y\in\Lambda$ and any $\delta>0$, there is a $\delta$-pseudo orbit $\{(x_i,t_i)\}_{i=1}^n(n>1)$ with all $x_i\in\Lambda$ such that $x_1=x$ and $x_n=y$.

Here we collect some generic properties of $C^1$ vector fields. Recall a subset $\mathcal{R}\subset \mathcal{X}^1(M)$ is called  {\it residual} if it contains an intersection of countably many open and dense subsets of $\mathcal{X}^1(M)$ and a property is called a \emph{generic} property if it holds in a residual set. Recall the definition of index in Section~\ref{Sec:introduction}.
\begin{Proposition}\label{Pro:generic}
There is a residual set $\mathcal{R}\subset \mathcal{X}^1(M)$ such
that for any $X\in\mathcal{R}$, the following properties are
satisfied:
\begin{enumerate}
\item $X$ is Kupka-Smale.
\item For any isolated transitive set $\Lambda$ with isolated neighborhood $U$, if $\Lambda$
contains a periodic orbit $P$, then $\Lambda$ equals to the
relative homoclinic class $H(P, U)$.
\item Let $\Lambda$ be a compact invariant set of $X$. If there is a sequence $X_n\to X$ and periodic orbits
$P_n$ of $X_n$ of index $i$ such that
$\lim_{n\to\infty}P_n=\Lambda$ in the Hausdorff topology, then
there exists a sequence of periodic orbits $Q_n^X$ of index $i$ of
$X$ itself such that $\lim_{n\to\infty}Q_n^X=\Lambda$ in the Hausdorff
topology.
\item For two open sets $U,V$ satisfying ${\overline U}\subset V$, if there are two hyperbolic periodic orbits $P_1,P_2\subset
U$ with ${\rm Ind}(P_1)<{\rm Ind}(P_2)$ such that for the
relative homoclinic classes, one has $H(P_1,U)=H(P_2,U)$,
then for any $i\in[{\rm Ind}(P_1),{\rm Ind}(P_2)]$, there
is a hyperbolic periodic orbit $P$ of index $i$ such that
$P_1, P_2\subset H(P,V)$.
\item Every chain transitive set of $X$ is a periodic limit.

\item An isolated chain transitive set $\Lambda$ of $X$ is robustly chain transitive, i.e., there exist a neighborhood $\mathcal{U}\subset\mathcal{X}^1(M)$ of $X$ and a neighborhood $U\subset M$ of $\Lambda$ such that for every $Y\in\mathcal{U}$, $$\Lambda_Y=\bigcap_{t\in \mathbb{R}}\varphi^Y_{t}(U)$$ is chain transitive.
\end{enumerate}

\end{Proposition}

Item 1 is from the classical Kupka-Smale theorem. Item 2 comes from a standard application of connecting lemma, one can see \cite{GW} for a proof. Item 3 is from the fact that hyperbolic periodic orbits are persistent and a standard generic argument, one can see \cite{Wen3} for a proof. Item 4 is a local version of the main result of \cite{ABCDW}. Item 5 is one of the main results of \cite{Cr06}. Item 6 is \cite[Corollaire 1.13]{BoC04}
\paragraph{Saddle value and Shilnikov bifurcation}

Let $\sigma$ be a hyperbolic singularity of $X$. Assume that the
eigenvalues of $DX(\sigma)$ can be arranged by the following:
$$Re(\lambda_m)<\cdots<Re(\lambda_1)<0<Re(\eta_1)<\cdots<Re(\eta_n).$$
Denote by $I(\sigma)=I(\sigma,X)=Re(\lambda_1)+Re(\eta_1)$. $I(\sigma)$ is said to be the \emph{saddle value} as in \cite{SSTC,ZGW08}.
 From the Shilnikov bifurcation theory, we have the following proposition.

\begin{Proposition}[\cite{SSTC}]\label{Pro:Shilnikov}
Let $X\in\mathcal{X}^1(M)$ and $\sigma$ be a singularity of $X$ with
saddle value $I(\sigma)<0$. If there is a homoclinic orbit $\Gamma$
of $\sigma$, then there exists an arbitrary small perturbation $Y$ of
$X$ such that $Y$ has a periodic orbit of index ${\rm Ind}(\sigma)$, which
is close to the homoclinic orbit $\Gamma$.
\end{Proposition}

One can change the sign of the saddle value under some conditions by small perturbations.

\begin{Lemma}\label{Lem:saddlevalueperturbation}
Given a $C^1$ vector field $X$, for any neighborhood $\mathcal{U}$ of $X$, there exist a neighborhood $\cV$ of $X$ and $\delta>0$ such that for any hyperbolic singularity $\sigma$ of $Z\in\cV$ with $|I(\sigma, Z)|<\delta$ and any neighborhood $U\subset M$ of $\sigma$, there is $Y\in\mathcal{U}$ such that $I(\sigma_Y, Y)<0$ and $Z(x)=Y(x)$ for any $x\in M\setminus U$.
\end{Lemma}

\begin{proof}
Let $b: [0,+\infty)\to[0,1]$ be the bump function with $b(x)=1$ for $x\in[0,1/3]$, $b(x)=0$ for $x\geq 1$ and $0\leq b'(x)<4$ for all $x\in[0,+\infty)$.

Given a neighborhood $\cU$ of $X$, there are a neighborhood $\cV$ of $X$ and $\varepsilon>0$ such that for any $Z\in\cV$, any vector field $Y$ which is $\varepsilon$ $C^1$-close to $Z$ is contained in $\cU$. Take $\delta=\varepsilon/10$.

Now we consider a vector field $Z\in\cV$ and a hyperbolic singularity $\sigma$ with $I(\sigma,Z)\in[0,\delta)$, one can have $Y\in\cU$ such that $I(\sigma_Y,Y)<0$. Now we give the construction of $Y$. In a local chart of $\sigma$, for which we identity $\sigma=0\in\RR^d$, for $a\in(-\varepsilon/5,0)$, for any arbitrarily small $r>0$, the expression of $Y$ is

$$ Y(y)=Z(y)+a\cdot b(\frac{\|(y)\|}{r})y, \forall y\in B_r(\sigma).$$
We have the following facts:
\begin{itemize}

\item $I(\sigma_Y,Y)=I(\sigma,Z)+2a$. Thus, there is some $a$ such that $\sigma_Y$ and $\sigma$ have the same index, and $I(\sigma_Y,Y)<0$.

\item The perturbation can be chosen in arbitrarily small neighborhood of $\sigma$ since $r$ can be chosen arbitrarily small independent of $a$ and $\varepsilon$.

\end{itemize}
The proof is now complete.
%
%
\end{proof}

\section{The periodic orbits in transitive sets}
In this section we will discuss the dominated splitting along
periodic orbits on normal bundle, then extend it to $\Lambda$. Let
$X\in\mathcal{X}^1(M)$, a sequence $(P_n,Y_n)$ is called an
$i$-fundamental sequence of $X$ if $Y_n$ converges $X$ in $C^1$
topology, the periodic orbit $P_n$ of $Y_n$ have index $i$ and
converge in the Hausdorff metric. The Hausdorff limit $\Gamma$ of
$P_n$ is called an \emph{$i$-periodic limit} of $X$. Note
that every $i$-periodic limit is a compact invariant set of $X$.

\begin{Lemma}\label{Lem:periodiclimitfromsingularity}
Assume that $\Lambda$ is a non-trivial transitive set. If $\Lambda$ contains a hyperbolic singularity $\sigma$ with saddle value $I(\sigma)\leq 0$, then $\Lambda$ contains an ${\rm Ind}(\sigma)$-periodic limit. Symmetrically, if $\Lambda$ contains a hyperbolic singularity $\sigma$ with saddle value $I(\sigma)\geq 0$, then $\Lambda$ contains an $ ({\rm Ind}(\sigma)-1)$-periodic limit.
\end{Lemma}
\begin{proof}
We only prove the lemma in the case of $I(\sigma)\leq 0$ since the case of $I(\sigma)\geq 0$ can be treated as the case that $I(\sigma)\leq 0$ after we reverse the vector field. Since $\Lambda$ is non-trivial and transitive, by Lemma \ref{Lem:nonemptyonmanifolds},  we can find $y_1\in (W^s(\sigma)\setminus\{\sigma\})\cap \Lambda$ and $y_2\in(W^u(\sigma)\setminus\{\sigma\})\cap\Lambda$.

Given any neighborhood $\mathcal{U}\subset\mathcal{X}^1(M)$ of $X$ and any neighborhood $U\subset M$ of $\Lambda$, we will see that there exists $Y\in\mathcal{U}$ such that $Y$ has a homoclinic orbit associated to $\sigma$ in $U$. This is a standard application of connecting lemma. Let $\mathcal{U}$ and $U$ be given neighborhoods. Then we can choose $L_1, T_1, \delta_{0,1}$ and $L_2, T_2, \delta_{0,2}$ be the triples of constants given in Lemma \ref{Lem:connectinglemma} corresponds to $z=y_1$ and $z=y_2$. Then we can take $\delta_1<\delta_{0,1}$ small enough such that the tube $\Delta_1=\cup_{t\in[0,T_1]}B(\varphi_t^X(y_1),\delta_1)\subset U$ and the the positive orbit of $y_1$ does not touch $\Delta_1$ once it leaves $\Delta_1$. Similarly, we can take $\delta_2<\delta_{0,2}$ small enough such that the tube $\Delta_2=\cup_{t\in[0,T_2]}B(\varphi_t^X(y_2),\delta_2)\subset U$ and the the negative orbit of $y_2$ does not touch $\Delta_2$ once it leaves $\Delta_2$. Without loss of generality, we can assume that $y_1, y_2$ in different orbits and $\Delta_1\cap\Delta_2=\emptyset$. Since $\Lambda$ is transitive, there exists $x\in\Lambda$ such that $\Lambda=\omega(x)$. So we can find orbit segment $\varphi_{[t_1, t_2]}(x)$ in $\Lambda$ such that $\varphi_{t_1}(x)\in B_{\delta_2/L_2}(y_2)$ and $\varphi_{t_2}(x)\in B_{\delta_1/L_1}(y_1)$.
Let $t_1<t_0<t_2$ be chosen such that $\varphi_{t_0}(x)\notin(\Delta_1\cup\Delta_2)$. Then we can apply the connecting lemma in $\Delta_1$ and $\Delta_2$ such that there is $Y\in \mathcal{U}$ such that $\varphi^X_{t_0}(x)$ is in the positive orbit of $y_2$ and the negative orbit of $y_1$ with respect to the flow $\varphi_t^Y$. Now we get a homoclinic orbit in $U$ associated to $\sigma$.

With an additional perturbation if necessary, we can assume that $I(\sigma)<0$ with respect to $Y$. Then by Proposition \ref{Pro:Shilnikov}, we can find a vector field $Z$ arbitrarily close to $Y$ such that $Z$ has a periodic orbit of index ${\rm Ind}(\sigma)$ in $U$. So for any neighborhood $\mathcal{U}$ of $X$ and any neighborhood $U$ of $\Lambda$, we can find $Z\in \mathcal{U}$ with a periodic orbit of index ${\rm Ind}(\sigma)$ in $U$. Hence we can find a sequence $Z_n\to X$ with periodic orbit $P_n$ of index ${\rm Ind}(\sigma)$ such that the Hausdorff limit of $\{P_n\}$ is contained in $\Lambda$. This ends the proof of the lemma.
\end{proof}

\begin{Lemma}\label{Lem:limit-dominated}
Assume that $X$ is away from homoclinic tangencies. If $\Lambda$ is an $i$-periodic limit, then
\begin{itemize}
\item $\cN_{\Lambda\setminus{\rm Sing}(X)}$ admits a dominated splitting of index $i$ with respect to $\psi_t$.

\item $\tilde\Lambda$ admits a dominated splitting of index $i$ with respect to $\widetilde\psi_t$.

\end{itemize}

\end{Lemma}
\begin{proof}
We assume that $X$ is far away from homolinic tangencies. Denote by
$$B^i(\Lambda)=\{e\in SM: \rho(e)\in\Lambda, \exists X_n\to X, \text{ and } p_n\in Per(X_n),$$
$$\ \ \ \ \ \ \ \ \ \ \ \ \ \ \ \ \  {\rm Ind}(p_n)=i, Orb_{X_n}(p_n)\subset U, \text{ such  that }
\frac{X_n(p_n)}{\|X_n(p_n)\|}\to e \}.$$
By the assumption that $\Lambda$ is an $i$-periodic limit, then the set
$$\{e=\frac{X(x)}{\|X(x)\|}:x\in\Lambda\setminus Sing(X)\}\subset B^i(\Lambda).$$
This implies $\tilde\Lambda\subset B^i(\Lambda)$. Let $e\in B^i(\Lambda)$ and $\{p_n\}$ be a sequence of hyperbolic periodic point of $X_n$ of index $i$ such that $X_n\to X$, $X_n(p_n)/\|X_n(p_n)\|\to e$ as $n\to\infty$. Denoted by $P_n$ the orbit of $p_n$ of $X_n$. Then from Proposition \ref{Pro:dominatedsplitting}, we know that for $n$ large, the individual hyperbolic splittings $N_{P_n}=N^s(P_n)\oplus N^u(P_n)$ of the usual linear Poincar\'e flow $\psi_t$ on $P_n$, put together, are $l$-dominated splitting. Now take limits $N^s(e)=\lim_{n\to\infty} N^s(p_n)$ and $N^u(e)=\lim_{n\to\infty}N^u(p_n)$ (here we take a subsequence such that the limits exist if necessary).  This gives a splitting $N_e=N^s(e)\oplus N^u(e)$ at every $e\in B^i(\Lambda)$ with the property
$$\frac{\|\tilde\psi_t|_{N^s(e)}\|}{m(\tilde\psi_{t}|_{N^u(e)})}\leq 1/2 \ \ \ \text{and}\ \ \ \ \frac{\|\tilde\psi_{-t}|_{N^u(e)}\|}{m(\tilde\psi_{-t}|_{N^s(e)})}\leq 1/2$$
for all $t\geq l$ where $m(A)$ denotes the mininorm of a linear map $A$. By the uniqueness of dominated splitting we know that $N_e=N^s(e)\oplus N^u(e)$ is uniquely determined and is $\tilde\psi_t$-variant hence we get a dominated splitting
$$N_{B^i(\Lambda)}=N^s\oplus N^u$$
of index $i$ on $B^i(\Lambda)$ with respect to $\tilde\psi_t$. In particular, $\tilde\Lambda$ have a dominated splitting of index $i$ with respect to $\tilde\psi_t$ and then $\Lambda\setminus Sing(X)$ have a dominated splitting of index $i$ with respect to $\psi_t$.
\end{proof}

\begin{Lemma}\label{Lem:splittingonperiodic}
There is a residual set $\mathcal{R}\subset{\cal X}^1(M)$ such that for any $X\in\mathcal{R}\setminus\overline{HT}$ and any isolated transitive set
$\Lambda$ of $X$, if $\Lambda$ contains an $i$-periodic limit, then there are a $C^1$ neighborhood
$\cal U$ of $X$, a neighborhood $U$ of $\Lambda$ and a constant
$l>0$ such that for any periodic orbit $P\subset U$ of
$Y\in\cal U$, $P$ admits an $l$-dominated splitting of index
$i$ in the normal bundle w.r.t. the linear Poincar\'e flow.
\end{Lemma}
\begin{proof} Let $\mathcal{R}$ be the residual set chosen in Proposition \ref{Pro:generic}. Let $X\in\mathcal{R}\setminus \overline{HT}$ and $\Lambda$ be an isolated transitive set of $X$. Note here that for a locally maximal invariant set of a $C^1$-generic vector field, transitivity is equivalent to chain transitivity by Item 3 and Item 5 of Proposition \ref{Pro:generic}. Thus $\Lambda$ is robustly chain transitive by Item 6 of Proposition \ref{Pro:generic}.

Let $\mathcal{U}_1\subset \mathcal{X}^1(M)$ and $U\subset M$ be the neighborhood of $X$ and $\Lambda$ such that for any $Y\in\mathcal{U}_1$, the maximal invariant set $\Lambda_Y$ in $U$ is chain transitive. Let $\mathcal{U}_2$ be the neighborhood of $X$ and constant $l$ be get from Proposition \ref{Pro:dominatedsplitting} such that for any $Y\in\mathcal{U}_2$ and any periodic orbit $P$ of $Y$, the hyperbolic splitting $N_{Q}=N^s(Q)\oplus N^u(Q)$ along $Q$ is an $l$-dominated splitting. Since $X$ contains an $i$-periodic limit we know that $X$ contains a hyperbolic periodic orbit $Q_X$ of index $i$ from Item 3 of Proposition \ref{Pro:generic}. Then we can take a neighborhood $\mathcal{U}_3$ of $X$ such that for any $Y\in \mathcal{U}_3$, there is a hyperbolic periodic orbit $Q_Y$  (close to $Q_X$) of index $i$ contained in $U$. We will prove that $\mathcal{U}=\mathcal{U}_1\cap\mathcal{U}_2\cap\mathcal{U}_3$, $U$ and $l$ satisfy the request of the lemma.

Let $Y\in \mathcal{U}$ and $P$ be a periodic orbit of $Y$ contained in $U$.  {~Since $X$ is far away from homoclinic tangencies,} by Lemma \ref{Lem:limit-dominated}, we just need to prove that $P$ is contained in an $i$-periodic limit of $Y$, {~then we will know that $P$ admits an $l$-dominated splitting of index $i$ in the normal bundle w.r.t. the linear Poincar\'e flow although the index of $P$ may not be $i$}. After a perturbation, we can assume $P$ is a hyperbolic periodic orbit for $Y_1$ where $Y_1$ is arbitrarily close to $Y$. After another arbitrarily small perturbation, we can find $Z\in\mathcal{R}$ arbitrarily close to $Y_1$. By Item 2 of Proposition \ref{Pro:generic} we know that $P_Z\subset \Lambda_Z=H(Q_Z, U)$ where $P_Z,Q_Z$ is the continuation of $P,Q_X$ w.r.t. $Z$, then we know that for any $\varepsilon>0$, there exists a periodic orbit $Q'$ of $Z$ of index $i$ such that $P_Z\subset B(Q',\varepsilon)$. Hence we can find a sequence of $Z_n$ with periodic orbits $Q'_n$ of $Z_n$ such that $Z_n\to Y$ and $P$ is contained in the Hausdorff limit of $Q'_n$. This proves that $P$ is contained in an $i$-periodic limit of $Y$. This ends the proof of the lemma.
\end{proof}

\section{The singularities in robustly transitive sets}

In this section, we will focus on the singularities in an isolated transitive set with dominated splittings for the linear Poincar\'e flow. At the end of this section, we prove Theorem \ref{Thm:uniqueindex}.

\begin{Lemma}\label{Lem:splittingonsingularity}

Let $\Lambda$ be a non-trivial transitive set. Assume that $\Lambda$ contains a
hyperbolic singularity $\sigma$ with hyperbolic splitting $E^s_\sigma\oplus E^u_\sigma$, and
\begin{itemize}

\item $\Lambda\setminus Sing(X)$ admits a dominated splitting
of index $i$ in the normal bundle w.r.t. the linear Poincar\'e flow
$\psi_t^X$;

\item ${\rm Ind}(\sigma)>i$.

\end{itemize}

Then $T_\sigma M$ admits a dominated splitting $T_\sigma
M=E^{ss}_\sigma\oplus E_\sigma^{cs}\oplus E_\sigma^u$ w.r.t. the tangent flow $\Phi_t^X$,
where $E_\sigma^{ss}$ is strongly contracting and $\dim E_\sigma^{ss}=i$.

\end{Lemma}

\begin{proof}
Assume $\Lambda\setminus Sing(X)$ admits a dominated splitting
$N_{\Lambda\setminus Sing(X)}M=N^{cs}\oplus N^{cu}$ with $\dim(N^{cs})=i$ w.r.t. the linear Poincar\'e flow $\psi^X_t$. Then we
can extend the dominated splitting to be
$N_{\tilde{\Lambda}}=N^{cs}(\tilde{\Lambda})\oplus
N^{cu}({\tilde{\Lambda}})$ w.r.t. the extended linear Poincar\'e flow $\tilde\psi_t$ as in Proposition~\ref{Pro:extendedLPF}.

By Lemma \ref{Lem:nonemptyonspace} we know that there exists an element $e\in\tilde\Lambda\cap E_\sigma^u$. Denote by $\pi_e$ the orthogonal projection from $T_\sigma M$ to $N_e$. Let $E_\sigma^{ss}=\pi_e^{-1}(N^{cs}(e))\cap E_\sigma^s$ and $E_\sigma^{cs}=\pi_e^{-1}(N^{cu}(e))\cap E^s_\sigma$. We are going to prove  that $E_\sigma^{ss}$ and $E_\sigma^{cs}$ are independent of the choice of $e$. 

Since $e\in E_\sigma^u$, we know that $\pi_e(E^s_{\sigma})$ is isomorphic to $E^s_\sigma$, then $E_\sigma^{ss}\oplus E_\sigma^{cs}=E_\sigma^s$. Since $\pi_e(E_\sigma^{ss})=N^{cs}(e)$, we have that $\dim(E_\sigma^{ss})=i$. Note that $\Phi_t(\pi_e^{-1}(N^{cs}(e))=\pi^{-1}_{\Phi_t^\#(e)}(N^{cs}(\Phi_t^\#(e)))$. Hence $\Phi_t(E_\sigma^{ss})=\pi_{\Phi_t^\#(e)}^{-1}(N^{cs}(\Phi_t^\#(e)))\cap E_\sigma^s$ by the fact that $E_\sigma^s$ is invariant for $\Phi_t$. Similarly, we have $\Phi_t(E_\sigma^{cs})=\pi_{\Phi_t^\#(e)}^{-1}(N^{cu}(\Phi_t^\#(e)))\cap E_\sigma^s$.

Now we proceed to prove that $E_\sigma^{ss}\oplus E_\sigma^{cs}$ is a dominated splitting with respect to tangent flow $\Phi_t$. Since $N^{cs}\oplus N^{cu}$ is a dominated splitting w.r.t. the extended linear Poincar\'e flow $\tilde\psi_t$, there exist
$C,\lambda>0$ such that for any unit vectors $u\in N^{cs}(e)$ and
$v\in N^{cu}(e)$,
$$\frac{|\tilde{\psi}_t(u)|}{|\tilde{\psi}_t(v)|}<Ce^{-\lambda t}$$
for any $t>0$. Note that $\Phi^\#_t(e)\in E_\sigma^u$ for any $t\in\mathbb{R}$, hence the angle between $\Phi_t^\#(e)$ and $E^s_\sigma$ has a positive lower bound. Thus there is a constant $K>1$ depending on the angle between $E_\sigma^s$ and $E^u_\sigma$
such that the orthogonal projection
$\pi_{\Phi_t^\#(e)}:E^s_\sigma\to N_{\Phi_t^\#(e)}$ has norm
$\|\pi^{-1}_{\Phi_t^\#(e)}\|\leq K$ for all $t\in\mathbb{R}$. Then for any unit vectors $u_{ss}\in E_\sigma^{ss}$ and $u_{cs}\in E_\sigma^{cs}$, we have
$$\frac{|\Phi_t(u_{ss})|}{|\Phi_t(u_{cs})|}=\frac{|\pi_{\Phi_t^\#(e)}^{-1}(\psi_t(\pi_e(u_{ss})))|}{|\pi_{\Phi_t^\#(e)}^{-1}(\psi_t(\pi_e(u_{cs})))|}\leq K\frac{|\psi_t(\pi_e(u_{ss}))|}{|\psi_t(\pi_e(u_{cs}))|}$$
\begin{equation}\ \ \ \ \ \ \ \ \ \ \ \ \ \ \ \ \ \ \ \ <KCe^{-\lambda t}\frac{|\pi_e(u_{ss})|}{|\pi_e(u_{cs})|}\leq KCe^{-\lambda t}\frac{1}{|\pi_e(u_{cs})|}<K^2Ce^{-\lambda t}\end{equation}
for all $t>0$. Similarly, we have
\begin{equation}\frac{|\Phi_{-t}(u_{cs})|}{|\Phi_{-t}(u_{ss})|}<K^2Ce^{-\lambda t}\end{equation}
for all $t>0$.
Hence $E_\sigma^s=E_\sigma^{ss}\oplus E_\sigma^{cs}$ is an invariant splitting and then is a dominated splitting from (1) and (2). From the uniqueness of dominated splitting, we can see that the splitting is independent of the choice of $e$. This ends the proof of Lemma~\ref{Lem:splittingonsingularity}.
\end{proof}

In the following lemmas we choose $\mathcal{R}$ to be the residual set given in Proposition \ref{Pro:generic} and Lemma \ref{Lem:splittingonperiodic}. As a direct consequence of Lemma~\ref{Lem:splittingonsingularity}, we have the following lemma.
\begin{Lemma}\label{Lem:splittingonsingularity2}
Let $X\in({\cal X}^1(M)\setminus\overline{HT})\cap\mathcal{R}$ and
$\Lambda$ be an isolated transitive set of $X$.  Assume that $\Lambda$ contains
an $i$-periodic limit. If $\sigma$ is a hyperbolic singularity in $\Lambda$, then we have the following cases:
\begin{itemize}
\item when ${\rm Ind}(\sigma)>i$, there is a dominated splitting $T_\sigma M=E^{ss}_\sigma\oplus E_\sigma^{cs}\oplus E^u_\sigma$ for the tangent flow with $\dim E^{ss}_\sigma=i$ and $E_\sigma^u$ is the unstable space of $\sigma$.
\item when ${\rm Ind}(\sigma)\le i$, there is a dominated splitting $T_\sigma M=E^{s}_\sigma\oplus E_\sigma^{cu}\oplus E^{uu}_\sigma$ for the tangent flow with $\dim E^{uu}_\sigma=\dim M-1-i$ and $E_\sigma^s$ is the stable space of $\sigma$.
\end{itemize}
\end{Lemma}
\begin{proof}
Since $\Lambda$ contains an $i$-periodic limit, from Item 3 of Proposition \ref{Pro:generic} we know that $\Lambda$ contains a periodic orbit $P$ of index $i$ and then $\Lambda$ is the relative homoclinic class of $P$ by the Item 2 of Proposition \ref{Pro:generic}. Hence $\Lambda$ is an $i$-periodic limit, then by Lemma \ref{Lem:limit-dominated} we know that there is a dominated splitting of index $i$ on $\Lambda\setminus Sing(X)$ with respect to the linear Poincar\'e flow. From Lemma \ref{Lem:splittingonsingularity}, for any given hyperbolic singularity $\sigma\in\Lambda$ with ${\rm Ind}(\sigma)>i$, we know that there is a dominated splitting $E_\sigma^{ss}\oplus E_\sigma^{cs}\oplus E_\sigma^u$ with respect to the tangent flow $\Phi_t$ with $\dim E_\sigma^{ss}=i$ and $E^{ss}_\sigma\oplus E^{cs}_\sigma$ is the stable space of $\sigma$. Symmetrically, for any hyperbolic singularity $\sigma\in\Lambda$ with ${\rm Ind}(\sigma)\le i$, there is a dominated splitting $E_\sigma^{s}\oplus E_\sigma^{cu}\oplus E_\sigma^{uu}$ with respect to tangent flow $\Phi_t$ with $\dim E_\sigma^{uu}=\dim M-1-i$.
\end{proof}

\begin{Lemma}\label{Lem:stableoutside}

Let $X\in{\cal X}^1(M)$. Assume that $\Lambda$ is a non-trivial transitive set
and contains a singularity $\sigma$. Suppose that we have the following properties:
\begin{itemize}
\item $T_\sigma M=E_\sigma^{ss}\oplus E_\sigma^c\oplus E_\sigma^{uu}$ is a partially
hyperbolic splitting on $\sigma$ w.r.t. the tangent flow,

\item There are $i\in[\dim E_\sigma^{ss}, \dim E_\sigma^{ss}\oplus E_\sigma^c-1]$,
$l>0$, and a $C^1$-neighborhood ${\cal U}$ of $X$ such that
every periodic orbit of $Y$ close to $\Lambda$ admits an
$l$-dominated splitting of index $i$ w.r.t. the linear
Poincar\'e flow $\psi_t^Y$.

\item For the corresponding strong unstable manifold
$W^{uu}(\sigma)$, one has
$W^{uu}(\sigma)\cap\Lambda\setminus\{\sigma\}\neq\emptyset$.

\end{itemize}

Then $W^{ss}(\sigma)\cap\Lambda=\{\sigma\}$.

\end{Lemma}

\begin{proof}
The proof is very similar to the proof of \cite[Lemma 4.3]{LGW}. For the completeness, we give the idea here. Without loss of generality, we assume that $E_\sigma^{ss}, E_\sigma^c, E_\sigma^{uu}$ are
mutually orthogonal. Suppose on the contrary that
$(W^{ss}(\sigma)\cap\Lambda)\setminus\{\sigma\}\neq \emptyset$. By
using the Connecting Lemma (Lemma~\ref{Lem:connectinglemma}), we can take a vector field $Y$ which is
arbitrarily close to $X$ such that $Y$ admits a strong homoclinic
connection $\Gamma\subset W^{ss}(\sigma)\cap W^{uu}(\sigma)$. By
another small perturbation, we can assume that the dynamics near
$\sigma$ is linear, then we have, the local strong stable manifold
$W^{ss}(\sigma)=\exp (E_\sigma^{ss}(\delta)$) and the local strong unstable
manifold $W^{uu}(\sigma)=\exp (E_\sigma^{uu}(\delta))$ for some small
$\delta$, where $E_\sigma^i(\delta)$ denote the disc in $E_\sigma^i$ with radius $\delta$ for $i=ss,uu$. Then we can get a vector field $Z$ arbitrarily close to
$Y$ with the following properties: $(a)$. $Z=Y$ in a neighborhood of $\sigma$,
$(b).$ $Z$ has a periodic orbit $P$ arbitrarily close to $\Gamma$
and $P$ has a segment located in $\exp(E^{ss}\oplus E^{uu})(\delta)$
for arbitrarily small $\delta$ (See the proof of \cite[Lemma 4.3]{LGW} for the details of
the perturbations).

Let $X_n\to X$ be a sequence of the perturbations given by the above
way and $P_n$ be the periodic orbit of $X_n$. Let $\Gamma$ be the
Hausdorff limit of $P_n$. From the second assumption of the Lemma,
we have an $l$-dominated splitting along the periodic orbit
$P_n$ with respect to the linear Poincar\'e flow of index $i$. Hence
there is a dominated splitting of index $i$ on $\Gamma\setminus Sing(X)$
with respect to the linear Poincar\'e flow and therefore there is a dominated splitting of index $i$ on $\tilde{\Gamma}$ for the extended linear
Poincar\'e flow. Denote by $N_{\tilde{\Gamma}}=N^{cs}_{\tilde{\Gamma}}\oplus
N^{cu}_{\tilde{\Gamma}}$ the dominated splitting of index $i$ with respect
to the extended linear Poincar\'e flow which generated by the
$l$-dominated splitting along the periodic orbits $P_n$. By the
assumptions of the Lemma, we know that $\dim(N^{cs}(e))\geq \dim(E^{ss})$
and $\dim(N^{cu}(e))\geq \dim(E^{uu})$ for all $e\in\tilde{K}$.

\bigskip

\begin{Claim}We have that
\begin{enumerate}
\item $E_\sigma^{uu}\subset N^{cu}(e)$ for any $e\in\tilde{\Gamma}\cap E_\sigma^{ss}$;
\item $E_\sigma^{ss}\subset N^{cs}(e)$ for any $e\in\tilde{\Gamma}\cap E_\sigma^{uu}$.
\end{enumerate}
\end{Claim}
\begin{proof}[The proof of Claim] We just prove item $1$. The item $2$ can be proved similarly. For any
$e\in\tilde{K}\cap E_\sigma^{ss}$, we have two splitting $N_e=N^{cs}(e)\oplus
N^{cu}(e)$ and $N_e=((E_\sigma^{ss}\oplus E_\sigma^c)\cap N_e)\oplus E_\sigma^{uu}$. Denote by $\Delta_e^{cs}=(E_\sigma^{ss}\oplus E_\sigma^c)\cap N_e$ and $\Delta_e^{cu}=E_\sigma^{uu}$. This gives a $\tilde\psi_t$-invariant splitting $N_{\tilde K\cap E_\sigma^{ss}}=\Delta^{cs}\oplus \Delta^{cu}$ on the $\Phi_t^\#$-invariant set $\tilde \Gamma\cap E_\sigma^{ss}$ (here we use the assumption $E_\sigma^{ss}$ is orthogonal to $E^{uu}$). We will
see that $N_{\tilde \Gamma\cap E^{ss}}=\Delta^{cs}\oplus \Delta^{cu}$ is a
dominated splitting with respect to the extended linear
Poincar\'e flow $\tilde\psi_t=\tilde{\psi}_t^X$. Since $(E_\sigma^{ss}\oplus E_\sigma^{c})\oplus E_\sigma^{uu}$ is a dominated splitting for $\Phi_t$, there exist $C\geq 1$ and $\lambda>0$ such that for any unit vectors $u\in(E_\sigma^{ss}\oplus E_\sigma^{c})$ and $v\in E_\sigma^{uu}$, one has
$$\frac{|\Phi_{t}(u)|}{|\Phi_{t}(v)|}<Ce^{-\lambda t}$$
for all $t>0$. For any $e\in \tilde \Gamma\cap E^{ss}$ and any unit vectors $u\in\Delta_e^{cs}$ and $v\in\Delta_e^{cu}$, we have
$$\frac{|\tilde\psi_{t}(u)|}{|\tilde\psi_{t}(v)|}=\frac{|\tilde\psi_{t}(u)|}{|\Phi_{t}(v)|}\leq\frac{|\Phi_{t}(u)|}{|\Phi_{t}(v)|}<Ce^{-\lambda t} $$
for all $t>0$. This proves that $\Delta^{cs}\oplus \Delta^{cu}$ is a dominated splitting for $\tilde\psi_t$. By the uniqueness of dominated splitting and the fact that $\dim(N^{cu})\geq \dim(\Delta^{cu})$ we have $E_\sigma^{uu}=\Delta_e^{cu}\subset N^{cu}(e)$ for any $e\in \tilde \Gamma\cap E^{ss}$. This ends the proof of Claim.
\end{proof}

Now let us continue the proof of Lemma \ref{Lem:stableoutside}. For every $P_n$, one can choose a point $p_n\in P_n\cap \exp_\sigma(E_\sigma^{ss}\oplus E_\sigma^{uu})$ such that
$\frac{X_n(p_n)}{\|X_n(p_n)\|}=v_n^{ss}+v_n^{uu}$ where $v_n^{ss}\in
E_\sigma^{ss}$, $v_n^{uu}\in E_\sigma^{uu}$ and $|v_n^{ss}|=|v_n^{uu}|$. By
choosing subsequences, we can assume that
$\frac{X_n(p_n)}{\|X_n(p_n)\|}$ accumulate to some point
$e\in\tilde{\Gamma}$. By the choice of $p_n$ we know that $e\in
E^{ss}\oplus E^{uu}$ and $e$ has a splitting $e=v^{ss}+v^{uu}$ such
that $v^{ss}\in E_\sigma^{ss}$, $v^{uu}\in E_\sigma^{uu}$ and $|v^{ss}|=|v^{uu}|$.
Denote by $e_t=\Phi_t^\#(e)$. One has that
$e_t\to E_\sigma^{uu}$ as $t\to +\infty$ and $e_t\to E_\sigma^{ss}$ as $t\to
-\infty$. Let $u=v^{ss}-v^{uu}\in N_e$ and
$u_t=\tilde{\psi}_t^X(u)$. A direct computation shows that
$$u_t=\Phi_t(u)-\langle\Phi_t(u),e_t\rangle e_t=\Phi_t(v^{ss}-v^{uu})-\frac{\|\Phi_t(v^{ss})\|^2-\|\Phi_t(v^{uu})\|^2}{\|\Phi_t(v^{ss})\|^2+\|\Phi_t(v^{uu})\|^2}\Phi_t(v^{ss}+v^{uu})$$
$$=2\frac{\|\Phi_t(v^{uu})\|^2\Phi_t(v^{ss})-\|\Phi_t(v^{ss})\|^2\Phi_t(v^{uu})}{\|\Phi_t(v^{ss})\|^2+\|\Phi_t(v^{uu})\|^2}:=u_t^{ss}+u_t^{uu}$$
where $u_t^{ss}\in E_\sigma^{ss}$ and $u_t^{uu}\in E_\sigma^{uu}$. By the domination, $\|u_t^{uu}\|/\|u_t^{ss}\|\to 0$ as $t\to+\infty$, hence $u_t/\|u_t\|\to E_\sigma^{ss}$ as
$t\to +\infty$. Similarly we have $u_t/\|u_t\|\to E_\sigma^{uu}$ as $t\to -\infty$.

Let $N_e=N^{cs}(e)\oplus N^{cu}(e)$ be the dominated splitting of index $i$
generated by the dominated splitting along $P_n$. Then we have two cases as following.

\paragraph{Case $1$: $u\in N^{cs}(e)$.} In this case we take a sequence $t_n\to-\infty$ such that $e_{t_n}$ converges to some $e'$ and $u_{t_n}/\|u_{t_n}\|$ converges to some $u'$. By the discussion in the previous paragraph, we know that $e'\in \tilde \Gamma\cap E_\sigma^{ss}$ and $u'\in E_\sigma^{uu}$. By the invariance and continuous of subbundle $N^{cs}$ we know $u'\in N^{cs}({e'})$. On the other hand we have $u'\in E_\sigma^{uu}\subset N^{cu}({e'})$ as $e'\in \tilde K\cap E_\sigma^{ss}$ from the claim. We get a contradiction, hence this case can not happen.
\paragraph{Case $2$: $u\notin N^{cs}(e)$.} In this case we take a sequence $t_n\to+\infty$ such that $e_{t_n}$ converges to some $e'$ and $u_{t_n}/\|u_{t_n}\|$ converges to some $u'$. By the discussion in the previous paragraph, we know that $e'\in \tilde \Gamma\cap E_\sigma^{uu}$ and $u'\in E_\sigma^{ss}$. By the fact that $N^{cs}\oplus N^{cu}$ is a dominated splitting and $u\notin N^{cs}(e)$ we know that $u'\in N^{cu}({e'})$. This is also a contradiction with $E_\sigma^{ss}\subset N^{cs}({e'})$ for any $e'\in \tilde\Gamma\cap E_\sigma^{uu}$.

Hence both two cases can not happen. The assumption $W^{ss}(\sigma)\cap\Lambda\setminus\{\sigma\}\neq \emptyset$ can not be true. This ends the proof of Lemma \ref{Lem:stableoutside}.
\end{proof}

Now let us consider a vector field $X\in({\cal X}^1(M)\setminus\overline{HT})\cap\mathcal{R}$. Assume $\Lambda$ is an isolated transitive set of $X$ and $\sigma\in\Lambda$ is a singularity. If $\Lambda$ contains an $i$-periodic limit, then by Lemma \ref{Lem:splittingonsingularity2} we know that if ${\rm Ind}(\sigma)>i$ then there is a partially splitting $T_\sigma M=E_\sigma^{ss}\oplus E_\sigma^{cs}\oplus E_{\sigma}^u$ with $\dim E_\sigma^{ss}=i$, if ${\rm Ind}(\sigma)\leq i$ then there is a partially hyperbolic splitting $T_\sigma M=E^{s}_\sigma\oplus E^{cu}_\sigma\oplus E^{uu}_\sigma$ with $\dim E_\sigma^{uu}=\dim M-1-i$. Denote by $W^{ss}(\sigma)$ and $W^{uu}(\sigma)$ the strong stable and unstable manifolds correspond to $E^{ss}_\sigma$ and $E^{uu}_\sigma$ in these two cases.

\begin{Lemma}\label{Lem:stableoutside2} There is a residual set $\cR$ with the following properties.
Let $X\in({\cal X}^1(M)\setminus\overline{HT})\cap\mathcal{R}$ and
$\Lambda$ be an isolated transitive set of $X$.  Assume that $\Lambda$ contains
an $i$-periodic limit. If $\sigma$ is a hyperbolic singularity in $\Lambda$, then we have the following cases:
\begin{itemize}
\item when ${\rm Ind}(\sigma)>i$, one has $(W^{ss}(\sigma)\setminus\{\sigma\})\cap\Lambda=\emptyset$.
\item when ${\rm Ind}(\sigma)\le i$, one has $(W^{uu}(\sigma)\setminus\{\sigma\})\cap\Lambda=\emptyset$.
\end{itemize}

\end{Lemma}
\begin{proof}
This is a direct corollary of Lemma~\ref{Lem:stableoutside}. Let $\sigma\in\Lambda$ be a hyperbolic singularity with ${\rm Ind}(\sigma)>i$. Since $\Lambda$ is a transitive set, from Lemma \ref{Lem:nonemptyonmanifolds} we have $(W^u(\sigma)\setminus \{\sigma\})\cap \Lambda\neq\emptyset$. Lemma \ref{Lem:splittingonperiodic} tells us that there is a $C^1$-neighborhood ${\cal U}$ of $X$ such that
every periodic orbit of $Y$ close to $\Lambda$ admits an
$l$-dominated splitting of index $i$ w.r.t. the linear
Poincar\'e flow $\psi_t^Y$. Note that $i=\dim E_\sigma^{ss}$ here. Then we can get that $(W^{ss}(\sigma)\setminus\{\sigma\})\cap\Lambda=\emptyset$ directly by Lemma \ref{Lem:stableoutside}. Similarly we can see that if ${\rm Ind}(\sigma)\le i$, then $(W^{ss}(\sigma)\setminus\{\sigma\})\cap\Lambda=\emptyset$.
\end{proof}

\begin{Lemma}\label{Lem:put-outside}
There is a residual set $\cR$ with the following properties. Let $X\in({\cal X}^1(M)\setminus\overline{HT})\cap\mathcal{R}$ and
$\Lambda$ be an isolated transitive set of $X$.  Assume that $\Lambda$ contains
an $i$-periodic limit. If $\sigma$ is a hyperbolic singularity in $\Lambda$, then we have the following cases:
\begin{itemize}
\item when ${\rm Ind}(\sigma)>i$, one has $B(\Lambda)\cap T_\sigma M\subset E_\sigma^{cs}\oplus E_\sigma^u$.
\item when ${\rm Ind}(\sigma)\le i$, one has $B(\Lambda)\cap T_\sigma M\subset E_\sigma^{s}\oplus E_\sigma^{cu}$.
\end{itemize}
\end{Lemma}
\begin{proof}
The proof is similar to the proof of \cite[Lemma 4.4]{LGW}. We give the sketch of proof here for completeness. Let $\sigma\in\Lambda$ be a hyperbolic singularity with ${\rm Ind}(\sigma)>i$. Once we have $e\in B(\Lambda)\setminus(E_\sigma^{cs}\oplus E_\sigma^u)$, then one can find an accumulated point of $\Phi_t^\#(e)$ in $E_\sigma^{ss}$, hence in $B(\Lambda)\cap E_\sigma^{ss}$. To prove $B(\Lambda)\cap T_\sigma M\subset (E_\sigma^{cs}\oplus E_\sigma^u)$, one just need to show that $B(\Lambda)\cap E_\sigma^{ss}=\emptyset$. Now we assume on the contrary, that is, there exists $e\in B(\Lambda)\cap E_\sigma^{ss}$. This means there are a sequence $X_n\to X$ and periodic points $p_n$ of $X_n$ whose orbit contained in a neighborhood $U$ (the isolation neighborhood of $\Lambda$) such that $X_n(p_n)/\|X_n(p_n)\|\to e$ as $n\to\infty$.

We can get a contradiction by the following way. Firstly, we can continuously extend the splitting $E^{ss}_\sigma\oplus(E_\sigma^{cs}\oplus E_\sigma^u)$ to a neighborhood of $\sigma$, and then put a cone field $C^{cu}$ in a neighborhood of $\sigma$ by
$$C_1^{cu}(x)=\{v=v^{ss}+v^{cu}\in T_x M: v^{ss}\in E^{ss}, v^{cu}\in E^{cs}\oplus E^u, \|v^{ss}\|<\|v^{cu}\|\}$$
at every $x$ close to $\sigma$. Note that $E^{ss}_\sigma\oplus(E_\sigma^{cs}\oplus E_\sigma^u)$ is a dominated splitting for $\Phi_t$ and the flow $\varphi_t^Y$ generated by $Y$ will be close to $\Phi_t^X$ in a small neighborhood of $\sigma$ once $Y$ close to $X$ enough in some local chart. Then we can take a neighborhood $\mathcal{U}$ of $X$ and a neighborhood $V(\sigma)$ of $\sigma$ and $T>0$ such that for any $Y\in \mathcal{U}$, once there is an orbit segment $\varphi_{[0,t]}^Y(x)$ of $Y$ contained in $V(\sigma)$ with $t>T$, then $\Phi_t^Y(C_1^{cu}(x))\subset C_1^{cu}(\varphi_t^Y(x))$ where $\varphi_t^Y$ and $\Phi_t^Y$ denote the flow and tangent flow generated by $Y$ as usual. Since $p_n\to\sigma$ and $\sigma$ is a hyperbolic fixed point, if necessary, take a subsequence of $p_n$, one can find a sequence $q_n=\varphi_{t_n}(p_n)$  with $t_n\to-\infty$ such that $q_n$ converges to some point $q$ in $W^s(\sigma)$ as $n\to\infty$. By applying Lemma \ref{Lem:stableoutside2} we have that $(W^{ss}(\sigma)\setminus\{\sigma\})\cap\Lambda=\emptyset$, then $q\notin W^{ss}(\sigma)$. Thus $X(\varphi_t(q))$ will accumulate in $E^{cs}$ as $t\to+\infty$. Now we can choose $T_1>0$ big enough such that $X(\varphi_{T_1}(q))\in C_1^{cu}(\varphi_{T_1}(q))$. Then we have $X_n(\varphi_{T_1}^{X_n}(q_n))\in C_1^{cu}(\varphi_{T_1}^{X_n}(q_n))$ for $n$ large enough. By the fact that $t_n\to -\infty$, we have $-t_n>T_1+T$ for $n$ large enough. Then we have
$$X_n(p_n)=X_n(\varphi^{X_n}_{-t_n}(q_n))=\Phi^{X_n}_{-t_n-T_1}(X_n(\varphi^{X_n}_{T_1}(q_n)))\in C_1^{cu}(p_n)$$
for $n$ large enough. Hence $X_n(p_n)/\|X_n(p_n)\|$ can not converge to any vector in $E^{ss}$. We have a contradiction. This proves that if ${\rm Ind}(\sigma)>i$, then $B(\Lambda)\cap T_\sigma M\subset E_\sigma^{cs}\oplus E_\sigma^u$. Similar proof holds for the case of ${\rm Ind}(\sigma)\le i$.
\end{proof}

In the following lemma, we assume that $X$ is locally linearizable near a singularity $\sigma$, that is, there is a small neighborhood $V$ of $\sigma$ and a diffeomorphism $\alpha:V\to \mathbb{R}^n$ such that $\alpha(\sigma)=0$ and $D_p\alpha(X(p))=A(\alpha(p))$ for every $p\in V$ where $A$ is a linear vector field on $\mathbb{R}^n$.

\begin{Lemma}\label{Lem:accumulation}
Assume that $\Lambda$ is a compact invariant set of $X$ and there are a
partially hyperbolic splitting $T_\sigma M=E_\sigma^{ss}\oplus E_\sigma^c\oplus
E_\sigma^{uu}$ w.r.t. the tangent flow and a dominated splitting $N_{\Lambda\setminus
Sing(x)}=N^{cs}\oplus N^{cu}$ w.r.t. the linear Poincar\'e flow with $\dim N^{cs}=\dim{E}^{ss}$. If
$W^{ss}(\sigma)\cap\Lambda=\{\sigma\}$, then for any $y\in
W^{uu}(\sigma)\setminus\{\sigma\}\cap\Lambda$ and any sequence
$y_n\in(\Lambda\setminus\{y\})\cap \exp_y(N_y)$ with
$\lim_{n\to\infty} y_n=y$, the set of accumulation points of
$$\left\{\frac{\exp_y^{-1}y_n}{\|\exp_y^{-1}y_n\|}\right\}$$
is contained in
$N_y^{cu}$.
\end{Lemma}
\begin{proof} The idea of this lemma is from \cite[Lemma 3.3]{ZGW08}. For any $x\in\Lambda\setminus Sing(X)$, define
$$D(x)=\{e\in T_xM: \exists x_n\in\Lambda, \text {s.t., } \frac{\exp_x^{-1}(x_n)}{\|\exp_x^{-1}(x_n)\|}\to e \text{ as } n\to\infty\},$$
$$D_N(x)=D(x)\cap N_x.$$
Then from \cite[Lemma 3.1]{ZGW08} we know that $D(x)$ is $\Phi_t^\#$ invariant. From Lemma 3.2 of \cite{ZGW08} we know that $D_N(x)$ is $\psi_t^\#$ invariant where $\psi_t^\#$ is defined by $\psi_t^\#(v)=\psi_t(v)/\|\psi_t(v)\|$ for all unit vector $v\in N$. The lemma here means that for any $y\in W^{uu}(\sigma)\setminus \{\sigma\}\cap\Lambda$, we have $D_N(y)\subset N^{cu}_y$.

Without loss of generality, we assume that $E_\sigma^{ss}$, $E_\sigma^c$ and $E_\sigma^{uu}$ are mutually orthogonal. The dominated splitting $N_{\Lambda\setminus
Sing(x)}=N^{cs}\oplus N^{cu}$ for the linear Poincar\'e flow $\psi_t$ can be extended to be a dominated splitting $N_{\tilde\Lambda}=N_{\tilde\Lambda}^{cs}\oplus N_{\tilde\Lambda}^{cu}$ for the extended linear Poincar\'e flow.  For any $y\in W^{uu}\setminus\{\sigma\}\cap \Lambda$ and any accumulation point $e$ of $\{X(\varphi_{t}(y))/\|X(\varphi_t(y))\|:t<0\}$, one has that $e\in E^{uu}_\sigma$. From the claim in the proof of Lemma \ref{Lem:stableoutside}, we know that $N_e^{cs}=E_\sigma^{ss}$ by the assumption $\dim N^{cs}=\dim E^{ss}_\sigma$.

By the assumption of locally linearizable of $X$ near $\sigma$, taking the local chart $(V,\alpha)$, ignoring the chart map $\alpha$, we can assume $\sigma=0\in\mathbb{R}^{n}=E_\sigma^{ss}\oplus E_\sigma^{c}\oplus E_\sigma^{uu}$,  $X(x)=(A^{ss}(x^{ss}), A^c(x^c), A^{uu}(x^{uu}))$ where $x=x^{ss}+x^c+x^{uu}$ and $x^{ss}\in E_\sigma^{ss}$, $x^c\in E_\sigma^c$, $x^{uu}\in E^{uu}$, $A^{ss}=A|_{E_\sigma^{ss}}$, $A^c=A|_{E_\sigma^c}$, $A^{uu}=A|_{E_\sigma^{uu}}$. Without loss of generality, we assume $\alpha(V)$ contains a box $$\{x^{ss}+x^{c}+x^{uu}|x^{i}\in E_\sigma^{i}, \|x^{i}\|\leq 2; i=ss,c,uu\}.$$ and $y=y^{uu}\in W_\sigma^{uu}$ with $\|y^{uu}\|=1$.
We can check that $N_y^{cu}=N_y\cap (E_y^{c}\oplus E_y^{uu} )$ where $E_y^c$ and $E_y^{uu}$ are the natural translation of $E_\sigma^c$ and $E_\sigma^{uu}$ respectively. Otherwise, if we have a vector $v\in N_y^{cu}$ such that $v=v^{ss}+v^{cu}$, $0\neq v^{ss}\in E_y^{ss}$ and $v^{cu}\in E_y^{c}\oplus E_y^{uu}$, then a direct computation shows that the directions of $\psi_t(v)$ will tend to the bundle $E_\sigma^{ss}$ as $t\to-\infty$, this contradicts with $N_e^{cs}=E_\sigma^{ss}$ for any accumulation point $e$ of $\{X(\varphi_{t}(y))/\|X(\varphi_t(y))\|:t<0\}$.

Now we suppose that $D_N(y)\subset N^{cu}_y$ does not hold. This means there is a vector $u\in D_N(y)\setminus N_y^{cu}$ with $u=u^{ss}+u^{cu}$ where $0\neq u^{ss}\in E_y^{ss}$ and $u^{cu}\in E_{y}^c\oplus E_{y}^{uu}$. Then there exists a sequence $y_n\to y$ as $n\to\infty$ such that
$$(y_n-y)/\|y_n-y\|=u_n^{ss}+u_n^{cu}\to u,$$
where $u_n^{ss}\in E^{ss}$ and $u_n^{cu}\in E^{c}\oplus E^{uu}$ . Write $a_n=\|y_n-y\|$, then $y_n=y+a_n(u_n^{ss}+u_n^{cu})$. Since $u^{ss}\neq 0$, we have $u_n^{ss}\neq 0$ for $n$ large enough, then we can take $t_n<0$ such that $\|{\rm exp}({t_n A})(a_nu_n^{ss})\|=1$. By the facts that $a_n\to 0$ and ${\rm exp}(tA)|_{E^{ss}}$ is contracting, we can see that $t_n\to-\infty$. By taking a subsequence if necessary, we can assume that $\|{\rm exp}(t_n A)(a_nu_n^{ss})\|\to z^{ss}$ as $n\to\infty$. Then by the fact that $E_\sigma^{ss}\oplus E_\sigma^{c}\oplus E_\sigma^{uu}$ is a dominated splitting for ${\rm exp}(tA)$, we have ${\rm exp}(t_nA)(y_n)\to z^{ss}$ as $n\to\infty$. Now we find a point $z^{ss}\in\Lambda\cap (W_\sigma^{ss}\setminus\{\sigma\})$. This is a contradiction.
\end{proof}

\begin{Lemma}\label{Lem:periodicinter}
Assume that $\Lambda$ is a robustly transitive set, which contains two
hyperbolic singularities $\sigma_1$ and $\sigma_2$ such that ${\rm
Ind}(\sigma_1)<{\rm Ind}(\sigma_2)$. Then for any $C^1$ neighborhood
${\cal U}$ of $X$, there is $Y\in \cal U$ which has a hyperbolic
periodic orbit $P\subset\Lambda_Y$ such that ${\rm
Ind}(P)\in[{\rm Ind}(\sigma_1),{\rm Ind}(\sigma_2)-1]$.

\end{Lemma}

\begin{proof}
If $I(\sigma_1)\leq 0$, then by Lemma \ref{Lem:periodiclimitfromsingularity} we can find $Y$ arbitrarily close to $X$ such that $Y$ has a hyperbolic periodic orbit $P$ with ${\rm Ind}(P)={\rm Ind}(\sigma_1)$. So the lemma is true in this case. Similarly, in the case of $I(\sigma_2)\ge0$ the lemma is also true.

Now we can assume that $I(\sigma_2)<0$ and $I(\sigma_1)>0$. Firstly, after applying Lemma \ref{Lem:periodiclimitfromsingularity} we can get a periodic orbit of index ${\rm Ind}(\sigma_1)-1$ after an arbitrarily small perturbation. Since $\Lambda$ is robustly transitive, apply Lemma \ref{Lem:periodiclimitfromsingularity} again for $\sigma_2$, we can get another periodic orbit of index ${\rm Ind}(\sigma_2)$ after another arbitrarily small perturbation. Thus for any neighborhood $\cal U$ of $X$, there
is $Z\in\cal U$ such that $\Lambda_Z$ has hyperbolic periodic orbits
$P_1$ of index ${\rm Ind}(\sigma_1)-1$ and $P_2$ of index
${\rm Ind}(\sigma_2)$. Since $Z$ can be accumulated by generic
vector fields, one can assume that $Z$ is in the residual set $\cal
R$ in Proposition \ref{Pro:generic}. Thus for an isolated neighborhood $U$, one has
$H(P_1,U)=H(P_2,U)=\Lambda_Z$ by Item 2 of
Proposition~\ref{Pro:generic}. Then by Item 4 of
Proposition~\ref{Pro:generic}, we get the conclusion.
\end{proof}

To prove Theorem \ref{Thm:uniqueindex}, we prove the following
theorem under some generic assumptions.

\begin{Theorem A'}
There is a residual set $\cR$ such that for any $X\in\cR$, if $\Lambda$ is a robustly transitive set of
$X\in{\cal X}^1(M)$, then either all the singularities in $\Lambda$
have the same index, or $X$ can be accumulated by vector fields with
a homoclinic tangency.

\end{Theorem A'}

\begin{proof}
We prove the theorem by contradiction. Assume $X\in
\mathcal{R}\setminus\overline{HT}$ and $\Lambda$ is a robustly
transitive set of $X$, and there exist two singularities $\sigma_1$
and $\sigma_2$ in $\Lambda$ with different indices. Without loss of generality,
we can assume ${\rm Ind}(\sigma_1)<{\rm Ind}(\sigma_2)$. By
Lemma~\ref{Lem:periodicinter}, there is $i\in[{\rm
Ind}(\sigma_1),{\rm Ind}(\sigma_2)-1]$ such that for any neighborhood $\mathcal{U}$ of $X$, there exist $Y\in\mathcal{U}$ and a hyperbolic
periodic orbit $P$ of $Y$ of index $i$ such that $P\subset\Lambda_Y$. As a corollary, $\Lambda$ contains an $i$-periodic limit. Note that a robustly transitive set is automatically an isolated transitive set. From Lemma \ref{Lem:splittingonsingularity2} we have
\begin{itemize}

\item $T_{\sigma_2}M$ admits a dominated splitting $E_{\sigma_2}^{ss}\oplus
E_{\sigma_2}^{cs}\oplus E_{\sigma_2}^u$ w.r.t. the tangent flow, where $E_{\sigma_2}^{ss}$ is strongly
contracting and $\dim E_{\sigma_2}^{ss}=i$.

\item $T_{\sigma_1}M$ admits a dominated splitting $E_{\sigma_1}^{s}\oplus
E_{\sigma_1}^{cu}\oplus E_{\sigma_1}^{uu}$ w.r.t. the tangent flow, where $E_{\sigma_1}^{uu}$ is strongly
expanding and $\dim E^{uu}=\dim M-i-1$.

\end{itemize}

Now by Lemma~\ref{Lem:stableoutside2}, we have that
$W^{ss}(\sigma_2)\cap \Lambda=\{\sigma_2\}$ and
$W^{uu}(\sigma_1)\cap \Lambda=\{\sigma_1\}$. Since the local strong stable and unstable manifolds of hyperbolic critical elements vary continuously and the robust transitive set $\Lambda_Y$ varies lower continuously on $Y$, there is a $C^1$
neighborhood $\mathcal{U}_0$ of $X$ such that for any $Y\in\mathcal{U}_0$, one has
$W^{ss}(\sigma_{2,Y})\cap \Lambda_Y=\{\sigma_{2,Y}\}$ and
$W^{uu}(\sigma_{1,Y})\cap \Lambda_Y=\{\sigma_{1,Y}\}$ where $\sigma_{1,Y}$ and $\sigma_{2,Y}$ are the continuations of $\sigma_1,\sigma_2$ with respect to $Y$.

For any $C^1$ neighborhood $\mathcal{U}\subset \mathcal{U}_0$ of $X$, by using the connecting
lemma (Lemma~\ref{Lem:connectinglemma}), there is $Y\in\cal U$ such that $W^s(\sigma_1)\cap
W^u(\sigma_2)\neq\emptyset$. With another arbitrarily small perturbation if necessary, we can assume that $Y$ is locally linearizable at $\sigma_1$ and $\sigma_2$. We have the following properties.
\begin{itemize}

\item By Lemma~\ref{Lem:transitive-periodic-likt} and Lemma~\ref{Lem:splittingonperiodic},  we have that $\Lambda_Y\setminus{\rm
Sing}(Y)$ admits a dominated splitting $N^{cs}\oplus N^{cu}$ of
index $i$ in the normal bundle w.r.t. the linear Poincar\'e flow.

\item Take $y\in W^s(\sigma_1)\cap
W^u(\sigma_2)$. Since $\Lambda_Y$ is transitive, there is a sequence
of points $\{y_n\}\subset \Lambda_Y$ with $y_n\neq y$ such that
$\lim_{n\to\infty}y_n=y$ and $\omega(y_n)=\Lambda_Y$ for each $n$.
Without loss of generality, we can assume that $y_n\in \exp_y(N_y)$
since $y$ is a regular point.

\end{itemize}

By Lemma~\ref{Lem:accumulation},  the non-empty set $D_N(y)$ of the accumulation
points of
$$\left\{\frac{\exp_y^{-1}y_n}{\|\exp_y^{-1}y_n\|}\right\}$$
has the following properties:
\begin{itemize}

\item The fact that $y\in W^{u}(\sigma_2)$ implies $D_N(y)\subset
N^{cu}(y)$.

\item The fact that $y\in W^{s}(\sigma_1)$ implies $D_N(y)\subset
N^{cs}(y)$.

\end{itemize}

This gives a contradiction and ends the proof of Theorem A'.
\end{proof}

\begin{proof}[The proof of Theorem~\ref{Thm:uniqueindex}]

First we notice that if $X$ is a vector field which is far away from
ones with a homoclinic tangency, and if $\Lambda$ is a robustly
transitive set, then all the hyperbolic singularity in $\Lambda$
should have the same index. This is because otherwise, there is a
vector field $Y\in\cal R$ which is close to $X$ such that
$\Lambda_Y$ is a robustly transitive set with singularities of
different indices. Then by using Theorem~A', we get a
contradiction.

Now if $\Lambda$ is a robustly transitive set which contains a
non-hyperbolic singularity $\sigma$. Without loss of generality, one
can assume that
\begin{itemize}
\item either, ${\rm D}X(\sigma)$ has only one eigenvalue whose real
part is zero,

\item or, ${\rm D}X(\sigma)$ has two eigenvalues whose real
parts are zero, which are conjugate.

\end{itemize}

In the first case, by a simple perturbation, we can have
singularities with different indices in the robustly transitive set,
which gives a contradiction.

Now we consider the second case. After a small
perturbation, we can get a vector field $Z\in\cR$ (here $\cR$ is given as in Lemma \ref{Lem:splittingonsingularity2}) that $\sigma$ is a hyperbolic
singularity of $Z$ whose saddle value is larger than 0 and
$T_\sigma M=E^{ss}\oplus E^{cs}\oplus E^u$, where $\dim E^{cs}=2$
and corresponds the complex eigenvalues. By Lemma~\ref{Lem:periodiclimitfromsingularity},
$\Lambda_Z$ contains an $({\rm
Ind}(\sigma)-1)$-periodic limit. Then by Lemma~\ref{Lem:splittingonsingularity2}, $E^{cs}$ should be divided into a dominated splitting, thus a contradiction. This ends the proof of Theorem \ref{Thm:uniqueindex}.
\end{proof}

{~We do some preparations for proving Theorem~\ref{Thm:improveZGW08}.
\begin{Lemma}\label{Lem:starimplies-farfrom}
If $\Lambda$ is a robustly transitive set, and $\Lambda$ is star, then there are a neighborhood $U$ of $\Lambda$ and a neighborhood $\cU$ of $X$ such that $Y\in\cU$ is far away from homoclinic tangencies in $U$. 
\end{Lemma}
\begin{proof}
A homoclinic tangency is in fact a minimally non-hyperbolic set of simple type in \cite{GW2}. Recall that a non-singular compact invariant set $\Gamma$ is said to be a \emph{minimally non-hyperbolic set} if $\Gamma$ is not hyperbolic, but any proper compact invariant subset of $\Gamma$ is hyperbolic. A minimally non-hyperbolic set $\Gamma$ is of \emph{simple} type, if there is a point $x\in\Gamma$ such that 
\begin{itemize}
\item $\omega(x)$ and $\alpha(x)$ are all proper subsets of $\Gamma$;
\item $N(x)=N^+(x)\oplus N^-(x)$ is not true.
\end{itemize}
Note that
$$N^\pm(x)=\{v\in N(x):~\lim_{t\to\pm\infty}\|\psi_t(v)\|=0\}.$$
By \cite[Lemma 6.1]{GW2}, if $\Lambda$ is star, then there are a neighborhood $U$ of $\Lambda$ and a neighborhood $\cU$ of $X$ such that $Y\in\cU$ has no minimally non-hyperbolic set of simple type in $U$. Hence $Y\in\cU$ has no homoclinic tangencies in $U$.
\end{proof}

}

\begin{proof}[The proof of Theorem~\ref{Thm:improveZGW08}]
By Lemma~\ref{Lem:starimplies-farfrom}, if $\Lambda$ is star, then there are a neighborhood $U$ of $\Lambda$ and a neighborhood $\cU$ of $X$ such that $X$ is far away from homoclinic tangencies in $U$. Thus by Theorem \ref{Thm:uniqueindex} we can see that if $\Lambda$ is robustly transitive and star, then all singularity in $\Lambda$ should be hyperbolic and have a common index (note here in the proof of Theorem \ref{Thm:uniqueindex} we just need consider the orbit near $\Lambda$).

\bigskip

It remains to show that $X$ is strongly homogenous.
\begin{Claim}
For any singularity $\sigma\in\Lambda$, the saddle value of $\sigma$ can not be zero.
\end{Claim}
\begin{proof}

 Assume on the contrary that there exists $\sigma\in\Lambda$ with $I(\sigma)=0$.

 Now we choose a neighborhood $\cU$ of $X$ such that $\Lambda$ has its robust continuation and star in $\cU$.
 For this $\cU$, one can find $\delta>0$ by Lemma~\ref{Lem:saddlevalueperturbation}.

 By Lemma \ref{Lem:periodiclimitfromsingularity} we can find a perturbation $Y_1$ close to $X$ such that
 \begin{itemize}
 \item $Y_1$ contains a hyperbolic periodic orbit $P$ with ${\rm Ind}(P)={\rm Ind}(\sigma)-1$.

 \item $I(\sigma_{Y_1},Y_1)$ is $\delta$-close to $I(\sigma,X)=0$, i.e., $|I(\sigma_{Y_1},Y_1)|<\delta$.

 \end{itemize}

For $Y_1$, we can make another perturbation in a small neighborhood of $\sigma$ by apply Lemma \ref{Lem:saddlevalueperturbation} to get $Y_2\in\mathcal{U}$ such that $P$ is still a periodic orbit of $Y_2$ with ${\rm Ind}(P)={\rm Ind}(\sigma)-1$ but the saddle value $I(\sigma)<0$. By applying Lemma \ref{Lem:periodiclimitfromsingularity} again, we can get a vector field $Y$ arbitrarily close to $Y_2$ with a periodic orbit $Q$ of index ${\rm Ind}(\sigma)$. If the perturbation is small enough, the continuation of $P$ exists with index ${\rm Ind}(\sigma)-1$. Thus we get a contradiction to the star condition by \cite[Lemma 1.6]{LGW}. We complete the proof of the Claim.
\end{proof}

Without loss of generality, now we assume that there exists a singularity $\sigma\in\Lambda$ with $I(\sigma)>0$. Then we can take a neighborhood $\mathcal{V}\subset\mathcal{U}$ such that for any $Z\in\mathcal{V}$, 
\begin{itemize}
\item the continuation $\sigma_Z$ of $\sigma$ with respect to $Z$ has the positive saddle value;

\item the set $\Lambda_Y$ is star for any $Y\in\cV$.
\end{itemize}

We will show that for every $Y\in\mathcal{V}$, the periodic orbit of $Y$ in $U$ should have index ${\rm Ind}(\sigma)-1$. We will prove by contradiction and assume that there exists $Y\in\mathcal{V}$ such that $Y$ has a hyperbolic periodic orbit $Q$ such that the index of $Q$ is not ${\rm Ind}(\sigma)-1$. 
One has that
\begin{itemize}
\item  By Lemma \ref{Lem:periodiclimitfromsingularity} we can find $Y'$ arbitrarily close to $Y$ such that $Y'$ contains a hyperbolic periodic orbit $P$ of index ${\rm Ind}(\sigma)-1$.

\item The continuation of $Q$ is still a hyperbolic periodic orbit of $Y'$ and the index of $Q$ is not ${\rm Ind}(\sigma)-1$.

\end{itemize}
%
One gets a contradiction by applying \cite[Lemma 1.6]{LGW} to the star set $\Lambda_{Y'}$. This ends the proof of Theorem \ref{Thm:improveZGW08}.
%
%
%
\end{proof}

\section{Multisingular partial hyperbolicity}\label{Sec:mutli-partial}

For compact invariant sets with singularities, we would like to study the weak form of hyperbolicity. However, it is not always reasonable to consider the tangent flow:
\begin{itemize}
\item By the suspension of a robust transitive diffeomorphism by Bonatti-Viana \cite{BoV00}, there is a robustly transitive vector field without any dominated splitting of the tangent flow.

\item By a recent example of Bonatti-da Luz, even for 5-dimensional vector field, there are star vector fields that cannot be characterized by singular hyperbolicity. They use multi-singular hyperbolicity \cite{BdL17} to characterize their example. The notion of multi-singular hyperbolicity is defined by using the linear Poincar\'e flow and the tangent flow along the vector field.

\end{itemize}

Thus, we need more general notion to understand vector fields that are away from homoclinic tangencies. Following the idea of multi-singular hyperbolicity by Bonatti-da Luz \cite{BdL17}, here we give a notion of {\it multisingular partially hyperbolicity}. Firstly we give the notion of reparametrizing cocycle following \cite{BdL17}. Let $\Lambda$ be a compact invariant set of vector field $X$. Let $\sigma\in\Lambda$ be a hyperbolic singularity of $X$, we say a positive continuous function $h: (\Lambda\setminus Sing(X))\times \mathbb{R}\to\mathbb{R}^+$ is a {\it pragmatical cocyle} (associated to $\sigma$) if $h(t+s,x)=h(t,x)\cdot h(s,\varphi_t(x))$ for all $s,t\in\mathbb{R}$ and $x\in\Lambda\setminus Sing(X)$ and there is an isolated neighborhood $V_\sigma$ of $\sigma$ such that the following conditions are satisfied:
\begin{enumerate}
\item if $x$ and $\varphi_t(x)$ are both in $V_\sigma$, then $h(x,t)=\|\Phi_t|_{<X(x)>}\|$;
\item if $x$ and $\varphi_t(x)$ are both outside $V_\sigma$, then $h(x,t)=1$.
\end{enumerate}
One can see \cite{BdL17} for the existence of a pragmatic cocycle associated to $\sigma$. A {\it reparametrizing cocycle } is a finite product of powers of pragmatic cocycles. As usual, we use $h_t(x)$ to denote $h(x,t)$ for a reparametrizing cocycle.

\begin{Definition}\label{Def:multisingular-partial}
Let $\Lambda$ be a compact invariant set. We say that $\Lambda$ is
\emph{multisingular partially hyperbolic}, if $N_{\Lambda\setminus{\
Sing}(X)}$ has a dominated splitting $N^{cs}\oplus N_1\oplus N^{cu}$
w.r.t. the linear Poincar\'e flow, and there exist two reparametrizing cocycles $h_t^s:\Lambda\setminus Sing(X)\to\mathbb{R}^+$ and $h_t^u:\Lambda\setminus Sing(X)\to\mathbb{R}^+$ such that,
\begin{enumerate}

\item $h_t^s\cdot\psi_t|_{N^{cs}}$ is uniformly contracting, that is, there are constants $C>0$ and $\lambda>0$ such that for any
regular point $x\in\Lambda$ and any $t>0$, one has
$$\|h_t^s(x)\cdot\psi_t|_{N^{cs}(x)}\|\leq C {\rm e}^{-\lambda t};$$

\item $h_t^u\cdot\psi_t|_{N^{cu}}$ is uniformly expanding, that is, there are constants $C>0$ and $\lambda>0$ such that for any
regular point $x\in\Lambda$ and any $t>0$, one has
$$\|h_{-t}^u(x)\cdot\psi_{-t}|_{N^{cu}(x)}\|\leq C {\rm e}^{-\lambda t}.$$
\end{enumerate}
\end{Definition}

\begin{Remark}
Here we do not use the extended linear Poincar\'e flow from \cite{LGW} which was used in \cite{BdL17}. The advantage by using extended linear Poincar\'e flow in \cite{BdL17} is that the notion is robust under  perturbation.

\end{Remark}

Here we verify that every non-trivial isolated transitive set is multisingular partially hyperbolic for a $C^1$ generic vector field away from ones with a homoclinic tangency.

\begin{Theorem}\label{Thm:multisingularpartiallyhyperbolic}
For any $X\in\mathcal{R}$ which is far away from ones with a homoclinic tangency,
every isolated non-trivial transitive set of $X$ admits a multisingular partially
hyperbolic splitting $N^{cs}\oplus N_1\oplus \cdots\oplus N_k\oplus
N^{cu}$, where $\dim N_i=1$ for any $1\le i\le k$.
\end{Theorem}

Before we give the proof of Theorem \ref{Thm:multisingularpartiallyhyperbolic}, we take some preparations. In the following, we will always assume $X\in\mathcal{R}$ and $\Lambda$ is  an isolated non-trivial transitive set of $X$.

Since $\Lambda$ is transitive, applying Lemma~\ref{Lem:transitive-periodic-likt}, $\Lambda$ is the Hausdorff limit of some hyperbolic periodic orbit
of some index. Assume that $\alpha$ and $\beta$ be the minimal and maximal number among the
indices. That is,
$$\alpha=\min\{i: \exists Y_n \text { and periodic orbit } \gamma_n \text{ of } Y_n,  \text{ s.t. }, Y_n\to Y,$$
$$ Ind(\gamma_n)=i \text{ and } \gamma_n \text{ converges into } \Lambda  \},$$
$$\beta=\max\{i: \exists Y_n \text { and periodic orbit } \gamma_n \text{ of } Y_n,  \text{ s.t. }, Y_n\to Y,$$
$$ Ind(\gamma_n)=i \text{ and } \gamma_n \text{ converges into } \Lambda  \}.$$

Because $\Lambda$ is nontrivial, $\Lambda$ contains no sinks, then from Item 3 of Proposition \ref{Pro:generic} we know that $\Lambda$ contains no $0$-periodic limit, hence we have $\alpha>0$.  Symmetrically, we have $\beta<\dim M-1$.

\begin{Lemma}\label{Lem:indexcomplete} There is a residual set $\cR\subset\cX^1(M)$ with the following properties.
Let $X\in\mathcal{R}\setminus\overline{HT}$ and $\Lambda$ be an isolated non-trivial transitive set of $X$ and $\alpha, \beta$ be given as above. Then for any integer $i\in[\alpha, \beta]$, there is a dominated slitting $N_{\Lambda\setminus Sing(X)}=N^1\oplus N^2$ (w.r.t. $\psi_t$) of index $i$. As a consequence, we have a dominated splitting $N_{\Lambda\setminus Sing(X)}=N^{cs}\oplus N_1\oplus\cdots\oplus N_{\beta-\alpha}\oplus N^{cu}$ with $\dim N^{cs}=\alpha$ and $\dim N^{cu}=d-1-\beta$ and $\dim N_i=1$ for any $1\le i\le \beta-\alpha$.
\end{Lemma}
\begin{proof}
 Note that $\Lambda$ is isolated with an isolated neighborhood $U$, this implies that if there is a periodic orbit in $U$, then this periodic orbit is contained in $\Lambda$. Thus by Item 3 of Proposition \ref{Pro:generic} we know that $\Lambda$ contains a periodic orbit $P_1$ of index $\alpha$ and contains a periodic orbit $P_2$ of index $\beta$. Then by Item 4 of Proposition \ref{Pro:generic} we know that for any $i\in[\alpha, \beta]$, there exists a periodic orbit $P$ of index $i$. Then by Item 2 of Proposition \ref{Pro:generic} we know that $\Lambda=H(P,U)$ hence an $i$-periodic limit. Thus we know that for any $i\in[\alpha, \beta]$, $\Lambda$ is an $i$-periodic limit. By Lemma~\ref{Lem:limit-dominated}, since $X$ is far away from ones with a homoclinic tangency, we know that for any $i\in[\alpha, \beta]$, there is a dominated slitting $N_{\Lambda\setminus Sing(X)}=N^1\oplus N^2$ of index $i$. As a consequence we have a dominated splitting $N_{\Lambda\setminus Sing(X)}=N^{cs}\oplus N_1\oplus\cdots\oplus N_k\oplus N^{cu}$ with $\dim N^{cs}=\alpha$ and $\dim N^{cu}=d-1-\beta$ and $\dim N_i=1$ for any $1\le i\le k$.
\end{proof}

\begin{Lemma}\label{Lem:uniformcontracting} There is a residual set $\cR\subset\cX^1(M)$ with the following properties.
Let $X\in\mathcal{R}\setminus\overline{HT}$ and $\Lambda$ be an isolated non-trivial transitive set of $X$ and $N^{cs},N^{cu}$ be the sub-bundles given in Lemma~\ref{Lem:indexcomplete}. Then there exist constants $l'>0, \eta'>0$ such that for any periodic point $p\in\Lambda$ of $X$, one has
\begin{equation*}\limsup_{n\to+\infty}\frac{1}{n}\sum_{i=0}^{n-1}\log \|\psi_{l'}|_{N^{cs}(\varphi_{il'}(p))}\|<-\eta'\end{equation*}
\begin{equation*}\limsup_{n\to+\infty}\frac{1}{n}\sum_{i=0}^{n-1}\log \|\psi_{-l'}|_{N^{cu}(\varphi_{-il'}(p))}\|<-\eta'\end{equation*}
\end{Lemma}

\begin{proof}
Let $P$ be a periodic orbit of $X$ in $\Lambda$ with hyperbolic splitting $N^s(P)\oplus N^u(P)$. By the choice of $\alpha$ we know that ${\rm Ind}(P)\geq \alpha$, then we have $N^{cs}(P)\subset N^s(P)$.
By Proposition \ref{Pro:dominatedsplitting2}, we can take $\eta>0$ and $N_0>0$ such that for any periodic orbit $P$ of $X$ with $\pi(P)\geq N_0$, one has
\begin{equation*}\prod_{i=0}^{[\pi(P)/l-1]}\|\psi_l|_{N^{cs}(\varphi_{il}(p))}\|< e^{-\eta [\pi(P)/l]},\end{equation*}
for any $p\in P$.

From the above estimation we know that for any periodic point $p\in\Lambda$ of $X$ with $\pi(p)\geq N_0$, one has
\begin{equation*}\limsup_{n\to+\infty}\frac{1}{n}\sum_{i=0}^{n-1}\log \|\psi_{l}|_{N^{cs}(\varphi_{il}(p))}\|<-\eta.\end{equation*}
This property still holds once we replace $l$ by one of it multiples. Hence we just need to find a positive integer $k$ and a constant $\eta'>0$ such that for any periodic point $p\in\Lambda$ of $X$ with $\pi(p)\leq N_0$, one has
\begin{equation*}\limsup_{n\to+\infty}\frac{1}{n}\sum_{i=0}^{n-1}\log \|\psi_{kl}|_{N^{cs}(\varphi_{ikl}(p))}\|<-\eta'\end{equation*}

Let
$$P_{N_0}(\Lambda)=\{p\in\Lambda: p \text{ is a periodic point with } \pi(p)\leq N_0\}.$$
Then $P_{N_0}(\Lambda)$ is a close subset of $\Lambda$ and hence a compact subset. For any $p\in P_{N_0}(\Lambda)$, since $\psi_t|_{N^{cs}({{\rm Orb}(p)})}$ is contracting (by the choice of $\alpha$), we can find $k(p)$ such that
$$\|\psi_{k(p)l}|_{N^{cs}(p)}\|\leq \frac{1}{2}.$$
 By the compactness of $P_{N_0}(\Lambda)$, we can find constants $\eta'>0$ and $C>1$ such that
$$\|\psi_{kl}|_{N^{cs}(p)}\|< C e^{-k\eta'}$$
for all $p\in P_{N_0}(\Lambda)$ and $k\geq 1$. Then we can find a large $k$ such that
$$\|\psi_{kl}|_{N^{cs}(p)}\|<  e^{-k\eta'}$$
for all $p\in P_{N_0}(\Lambda)$. Then we can see that $l'=kl$ and $\eta'$ satisfy the requirements of the lemma. Symmetrically, we can take $l'$ and $\eta'$ such that the estimation for the bundle $N^{cu}$ is satisfied.
\end{proof}

To simplify the notations, we still use $l,\eta$ to denote $l',\eta'$ in the above lemma since we can choose $l',\eta'$ keep all properties of $l,\eta$ we have listed in Proposition \ref{Pro:dominatedsplitting2}.

Now we analyse the singularities in $\Lambda$.

\begin{Lemma}\label{Lem:indicesofsingularity}
There is a residual set $\cR\subset\cX^1(M)$ with the following properties. Let $X\in\mathcal{R}$ and $\Lambda$ be an isolated non-trivial transitive set of $X$ and $\alpha, \beta$ be given as above. Then for any singularity $\sigma\in\Lambda$, we have $\alpha\leq {\rm Ind}(\sigma)\leq\beta+1$. Furthermore, if ${\rm Ind}(\sigma)=\alpha$ then $I(\sigma)<0$; if ${\rm Ind}(\sigma)=\beta+1$, then $I(\sigma)>0$.
\end{Lemma}

\begin{proof}
This is a direct consequence of Lemma \ref{Lem:periodiclimitfromsingularity}. Let $\sigma\in\Lambda$ be a singularity. By Lemma \ref{Lem:periodiclimitfromsingularity} we know that $\Lambda$ contains an ${\rm Ind}(\alpha)$-periodic limit in the case of $ I(\sigma)\leq 0$ and contains an $({\rm Ind}(\alpha)-1)$-periodic limit in the case of $I(\sigma)\geq 0$. By the choice of $\alpha$ we know that ${\rm Ind}(\sigma)\geq\alpha$ and if ${\rm Ind}(\sigma)=\alpha$ then $I(\sigma)<0$. Symmetrically we have ${\rm Ind}(\sigma)\leq \beta+1$ and if ${\rm Ind}(\sigma)=\beta+1$ then $I(\sigma)>0$ by the choice of $\beta$. This ends the proof of the lemma.
\end{proof}

Recall that $\tilde\Lambda$ denotes the closure of $\{X(x)/\|X(x)\|: x\in\Lambda\setminus Sing(X)\}$ in $SM$. Given a singularity $\sigma\in\Lambda$, denote by $\tilde\Lambda_\sigma=\tilde\Lambda\cap T_\sigma M$. Since both $\tilde\Lambda$ and $T_\sigma M$ are $\Phi_t^\#$-invariant, we know $\tilde\Lambda_\sigma$ is invariant under $\Phi_t^\#$.

\begin{Lemma}\label{Lem:tildelambdaatsingularity}
There is a residual set $\cR\subset\cX^1(M)$ with the following properties. Let $X\in\mathcal{R}\setminus\overline{HT}$ and $\Lambda$ be an isolated non-trivial transitive set of $X$ and $\alpha, \beta$ be given as above.

Then for any singularity $\sigma\in\Lambda$, there is a partially hyperbolic splitting $T_\sigma M=E^{ss}_\sigma\oplus E_\sigma^1\oplus E_\sigma^2\oplus\cdots\oplus E_\sigma^{\beta+1-\alpha}\oplus E^{uu}_\sigma$ (w.r.t. $\Phi_t$), where $E^{ss}_\sigma$ is contracting with $\dim E_\sigma^{ss}=\alpha$, $E^{uu}_\sigma$ is expanding with $\dim E_{\sigma}^{uu}=\dim M-\beta-1$ and $\dim E_\sigma^i=1$ for all $i=1, 2, \cdots, \beta+1-\alpha$.

Moreover, for $\tilde\Lambda_\sigma$, one has that

\begin{enumerate}
\item  if ${\rm Ind}(\sigma)=\alpha$ then $\tilde\Lambda_\sigma\subset E^{ss}_\sigma\oplus E_\sigma^1$;

\item if ${\rm Ind}(\sigma)>\alpha$ then $\tilde\Lambda_\sigma\cap E_\sigma^{ss}=\emptyset$.

\end{enumerate}

\end{Lemma}

\begin{proof}
Let $\sigma\in\Lambda$ be a singularity with a hyperbolic splitting $E_\sigma^s\oplus E_\sigma^u$. By Lemma \ref{Lem:indicesofsingularity} we know that ${\rm Ind}(\sigma)\in[\alpha,\beta+1]$. Since for all $i\in[\alpha,\beta]$, $\Lambda$ contains $i$-periodic limit,  by Lemma \ref{Lem:splittingonsingularity2} we know that (1) for any $\alpha\leq i<{\rm Ind}(\sigma)$, there exists a dominated splitting $T_\sigma M=E\oplus F$ of index $i$ with $E\subset E_\sigma^s$; (2) for any ${\rm Ind}(\sigma)<i\leq\beta+1$, there exists a dominated splitting $T_\sigma M=E\oplus F$ of index $i$ with $F\subset E_\sigma^u$. Note that $E^s_\sigma\oplus E_\sigma^u$ is already a dominated splitting of index ${\rm Ind}(\sigma)$. Hence for all $\alpha\leq i\leq \beta+1$, there is a dominated splitting $T_\sigma M=E\oplus F$ of index $i$, then we have a dominated splitting $T_\sigma M=E^{ss}_\sigma\oplus E_\sigma^1\oplus E_\sigma^2\oplus\cdots\oplus E_\sigma^{\beta+1-\alpha}\oplus E^{uu}_\sigma$ with the properties: (1) $E_\sigma^{ss}\subset E_\sigma^s$ and $E^{uu}_\sigma\subset E_\sigma^u$; (2) $\dim E_\sigma^{ss}=\alpha$ and $\dim E_{\sigma}^{uu}=\dim M-\beta-1$; (3) $\dim E_\sigma^i=1$ for all $i=1, 2, \cdots, \beta+1-\alpha$.
%
%
%
%

Let $\sigma\in\Lambda$ be a singularity of $X$. If ${\rm Ind}(\sigma)=\alpha$, then by the fact that $\Lambda$ contains an $\alpha$-periodic limit we know $\tilde\Lambda\subset B(\Lambda)\subset E_\sigma^{ss}\oplus E_\sigma^1$ from Lemma~\ref{Lem:wildecontainedinB} and Lemma~\ref{Lem:put-outside}. If ${\rm Ind}(\sigma)>\alpha$, also by the fact that $\Lambda$ contains an $\alpha$-periodic limit we know that $\tilde\Lambda\subset B(\Lambda)\subset E_\sigma^1\oplus E_\sigma^2\oplus\cdots\oplus E_\sigma^{\beta+1-\alpha}\oplus E^{uu}_\sigma$, hence $\tilde\Lambda\cap E_\sigma^{ss}=\emptyset$. This ends the proof of the lemma.
\end{proof}

By Proposition \ref{Pro:extendedLPF} we know that the dominated splitting $N_{\Lambda\setminus Sing(X)}=N^{cs}\oplus N_1\oplus\cdots\oplus N_{\beta-\alpha}\oplus N^{cu}$ on $\Lambda\setminus Sing(X)$ (w.r.t. $\psi_t$) can be extended to be a dominated splitting $N_{\tilde\Lambda}=N^{cs}\oplus N_1\oplus\cdots\oplus N_{\beta-\alpha}\oplus N^{cu}$ on $\tilde\Lambda$ with respect to the extended linear Poincar\'e flow $\tilde\psi_t$.

\begin{Lemma}\label{Lem:contractingatsingularity}
There is a residual set $\cR\subset\cX^1(M)$ with the following properties. Let $X\in\mathcal{R}\setminus\overline{HT}$ and $\Lambda$ be an isolated non-trivial transitive set of $X$ and $\sigma$ be a singularity in $\Lambda$. If ${\rm Ind}(\sigma)=\alpha$, then there exist $C>1$, $\lambda>0$ such that
$$\|\Phi_t(e)\|\cdot \|\tilde\psi_t|_{N^{cs}(e)}\|< C{\rm e}^{-\lambda t}$$
for all $e\in \tilde\Lambda_\sigma$ and $t>0$. If ${\rm Ind}(\sigma)>\alpha$, then there exist $C>1$, $\lambda>0$ such that
$$\|\Phi_t(e)\|^{-1}\cdot \|\tilde\psi_t|_{N^{cs}(e)}\|< C{\rm e}^{-\lambda t}$$
for all $e\in \tilde\Lambda_\sigma$ and $t>0$.
\end{Lemma}
\begin{proof}
Assume $\sigma\in\Lambda$ is a singularity with ${\rm Ind}(\sigma)=\alpha$. From Lemma \ref{Lem:tildelambdaatsingularity} we know that $\tilde\Lambda_\sigma\subset E_\sigma^{ss}\oplus E_\sigma^1$. Denote by $F^{cs}_\sigma=E_\sigma^{ss}\oplus E_\sigma^1$ and $F^u_\sigma=E_\sigma^2\oplus\cdots\oplus E_\sigma^{\beta+1-\alpha}\oplus E^{uu}_\sigma$. Let $\pi_e$ be the orthogonal projection from $T_\sigma M$ to $N_e$ for any $e\in\tilde\Lambda_\sigma$.

\begin{Claim}
$N^{cs}(e)=\pi_e(F^{cs}_\sigma)$ for all $e\in\tilde\Lambda_\sigma$.
\end{Claim}
\begin{proof}[Proof of the Claim]
We define two sub-bundles $\Delta^{cs}, \Delta^{cu}\subset N_{\tilde\Lambda_\sigma}$ by letting
$$\Delta^{cs}(e)=\pi_e(F^{cs}_\sigma),~~~ \Delta^{cu}(e)=\pi_e(F^u_\sigma), ~~~\forall e\in\tilde\Lambda_\sigma.$$
By the fact that $T_\sigma M=F^{cs}_\sigma\oplus F^u_\sigma$ we know that $N_{\tilde\Lambda_\sigma}=\Delta^{cs}\oplus \Delta^{cu}$. Since $F^{cs}_\sigma$ and $F^u_\sigma$ are invariant by $\Phi_t$, we can get that $\Delta^{cs}$ and $\Delta^{cu}$ are invariant by $\tilde\psi_t$, that is, $$\tilde\psi_t(\Delta^{cs}(e))=\Delta^{cs}(\Phi_t^\#(e)), ~~~\tilde\psi_t(\Delta^{cu}(e))=\Delta^{cu}(\Phi_t^\#(e)), ~~~\forall e\in\tilde\Lambda_\sigma, ~t\in\mathbb{R}.  $$
Since $e\in F^{cs}_\sigma$ we know that there is a lower bound for the angle between $e\in\tilde\Lambda_\sigma$ and $F^{u}_\sigma$, hence there is $K>1$ such that $m(\pi_e|_{F^{u}_\sigma})\geq K^{-1}$ for all $e\in\tilde\Lambda_\sigma$.
For any $e\in\tilde\Lambda_\sigma$ and any unit vectors $u\in\Delta^{cs}(e), v\in\Delta^{cu}(e)$, letting $v=\pi_e(v')$ where $v'\in F^u_\sigma$, we have

$$\frac{\|\tilde\psi_{t}(u)\|}{\|\tilde\psi_{t}(v)\|}\leq\frac{\|\Phi_{t}(u)\|}{\|\tilde\psi_{t}(v)\|}=\frac{\|\Phi_{t}(u)\|}{\|\pi_{\Phi^\#_t(e)}(\Phi_{t}(v'))\|}\leq K\frac{\|\Phi_{t}(u)\|}{\|\Phi_{t}(v')\|}\leq K\frac{\|\Phi_t|_{F^{cs}_\sigma}\|}{m(\Phi_t|_{F^u_\sigma})}.$$

Note that $F^{cs}_\sigma\oplus F^u_\sigma$ is a dominated splitting of index $\alpha+1$. From the above estimation we can see that $\Delta^{cs}\oplus \Delta^{cu}$ is a dominated splitting. Since $e\in F^{cs}_\sigma$, we know that $\dim \Delta^{cs}=\dim F^{cs}_\sigma=\alpha$. Note that $N^{cs}\oplus (N_1\oplus\cdots\oplus N_{\beta-\alpha}\oplus N^{cu})$ is also a dominated splitting of index $\alpha$. We can see that $\Delta^{cs}=N^{cs}$ by the uniqueness of dominated splitting. Hence we have $N^{cs}(e)=\pi_e(F^{cs}_\sigma)$ for all $e\in\tilde\Lambda_\sigma$.
\end{proof}

Note that $E_\sigma^{ss}$ is the stable space of $\sigma$ and $E_\sigma^1$ is the eigenspace associated to the weakest unstable eigenvalue of $\Phi_1$. From Lemma \ref{Lem:indicesofsingularity} we know that
$$\lim_{t\to+\infty}\frac{1}{t}\log\|\wedge^2\Phi_t|_{F^{cs}_\sigma}\|=I(\sigma)<0.$$
Hence there are $C>1,\lambda>0$ such that
$$\|\wedge^2\Phi_t|_{F^{cs}_\sigma}\|<C{\rm e}^{-\lambda t}$$
for all $t>0$. For any $e\in\tilde\Lambda_\sigma$ and any unit vector $u\in N^{cs}(e)$, we have
$$\|\Phi_t(e)\|\cdot\|\tilde\psi_t(u)\|=\|\Phi_t(e)\wedge\Phi_t(u)\|=\|\wedge^2\Phi_t(e\wedge u)\|<C{\rm e}^{-\lambda t}$$
for all $t>0$. Hence we have $$\|\Phi_t(e)\|\cdot \|\tilde\psi_t|_{N^{cs}(e)}\|< C{\rm e}^{\lambda t}$$
for all $e\in \tilde\Lambda_\sigma$ and $t>0$.

Now let us assume $\sigma\in\Lambda$ is a singularity with ${\rm Ind}(\sigma)>\alpha$. In this case we know that $\tilde\Lambda_\sigma\cap E^{ss}_\sigma=\emptyset$. Similar as in above, by the uniqueness of dominated splitting,  we can prove that $N^{cs}(e)=\pi_e(E_\sigma^{ss})$ for any $e\in\tilde\Lambda_\sigma$, where $\pi_e$ denotes the orthogonal projection from $T_\sigma M$ to $N_e$. Note that $\tilde\Lambda_\sigma\subset E_\sigma^1\oplus \cdots E_\sigma^{\beta+1-\alpha}\oplus E^{uu}_\sigma$, we have a constant $K'>1$ such that $m(\pi_e|_{E_\sigma^{ss}})>K'^{-1}$ for all $e\in\tilde\Lambda_\sigma$.  Then by the fact that $E^{ss}_\sigma\oplus (E_\sigma^1\oplus \cdots E_\sigma^{\beta+1-\alpha}\oplus E^{uu}_\sigma)$ is a dominated splitting we know that there are $C_0>1$ and $\lambda>0$ such that for any $u\in E_\sigma^{ss}$ and $v\in E_\sigma^1\oplus \cdots E_\sigma^{\beta+1-\alpha}\oplus E^{uu}_\sigma$ and $t>0$, one has
$$\frac{\|\Phi_t(u)\|}{\|\Phi_t(v)\|}<C_0 {\rm e}^{-\lambda t}.$$
Hence for any $e\in\tilde\Lambda_\sigma$ and any unit vector $u\in N^{cs}(e)$ and $t>0$, letting $u=\pi_e(u')$ where $u'\in E_\sigma^{ss}$, we have
$$\frac{\|\tilde\psi_t(u)\|}{\|\Phi_t(e)\|}=\frac{\|\pi_{\Phi_t^\#(e)}(\Phi_t(u'))\|}{\|\Phi_t(e)\|}\leq\frac{\|\Phi_t(u')\|}{\|\Phi_t(e)\|}\leq K'C_0{\rm e}^{-\lambda t}.$$
Let $C=K'C_0$, then for any $e\in\tilde\Lambda_\sigma$ and $t>0$, we have
$$\|\Phi_t(e)\|^{-1}\cdot \|\tilde\psi_t|_{N^{cs}(e)}\|< C{\rm e}^{\lambda t}.$$
This ends the proof of Lemma \ref{Lem:contractingatsingularity}.
\end{proof}

No we proceed to prove Theorem~\ref{Thm:multisingularpartiallyhyperbolic}.
\begin{proof}[Proof of Theorem~\ref{Thm:multisingularpartiallyhyperbolic}]
To prove $\Lambda$ is multisingular partially hyperbolic, we will find two reparemtrizing cocycles $h_t^s$ and $h_t^u$ such that $h_t^s\cdot \psi_t|_{N^{cs}}$ is contracting and $h_t^u\cdot\psi_t|_{N^{cu}}$ is expanding. Here we construct $h_t^s, h_t^u$ as following. Let $S$ be the set of singularities of $X$ contained in $\Lambda$ and denote by
$$S_1=\{\sigma\in\Lambda: \sigma \text{ is a singularity of index } \alpha\},$$
$$S_2=\{\sigma\in\Lambda: \sigma \text{ is a singularity of index } \beta+1\}.$$
Since $X$ is a Kupka-Smale system we know that $X$ contains only finitely many singularities. Write
$$S_1=\{\sigma_1,\sigma_2, \cdots, \sigma_k\},$$
$$S_2=\{\sigma_1', \sigma_2', \cdots, \sigma_l'\}.$$
Now we choose $c>0$ small enough such that for any $\sigma\in S$, the connected component $V_\sigma$ of $\{x\in M: \|X(x)\|<c\}$ containing $\sigma$ is an isolated neighborhood for $\sigma$. Now we define a positive continuous functions $k^\sigma:\Lambda\setminus Sing(X)\to\mathbb{R}$ associated to $\sigma$ by the following way:
$$k^{\sigma}(x)=\|X(x)\|, \text{ when } x\in V_\sigma;~~~ k^\sigma(x)=c, \text{ when } x\notin V_\sigma;$$
Then define a pragmatic cocycle $h_t^\sigma$ associated to the singularity $\sigma$ by
$$h_t^\sigma(x)=\frac{k^\sigma(\varphi_t(x))}{k^\sigma(x)}, ~\forall x\in\Lambda\setminus Sing(x), t\in\mathbb{R};$$
Now we can define two reparametrizing cocycles $h_t^s, h_t^u$ by letting
$$h_t^s(x)=(\prod_{\sigma\in S_1}h_t^{\sigma}(x))\cdot(\prod_{\sigma\in S\setminus S_1}h_t^{\sigma}(x))^{-1},~~~~\forall x\in\Lambda\setminus Sing(X),~ t\in\mathbb{R};$$
$$h_t^u(x)=(\prod_{\sigma\in S_2}h_t^{\sigma}(x))\cdot(\prod_{\sigma\in S\setminus S_2}h_t^{\sigma}(x))^{-1},~~~~\forall x\in\Lambda\setminus Sing(X),~ t\in\mathbb{R}.$$

We will prove that  $h_t^s\cdot \psi_t|_{N^{cs}}$ is contracting. Note that for any singularity $\sigma\in\Lambda$, the pragmatic cocycle $h_t^\sigma$ associated to $\sigma$ gives automatically a cocycle $\tilde h_t^\sigma$ on $\{X(x)/\|X(x)\|:x\in\Lambda\setminus Sing(X)\}$ (with respect to the flow $\Phi_t^\#$) by letting
$$\tilde h_t^\sigma(X(x)/\|X(x)\|)=h_t^\sigma(x),~~\forall x\in\Lambda\setminus Sing(X).$$
$\tilde h_t^\sigma$ can be continuously extended to $\tilde\Lambda$ as following:
$$\tilde h_t^{\sigma}(e)=\|\Phi_t(e)\|, ~~\forall e\in\tilde\Lambda_\sigma,$$
$$\tilde h_t^\sigma(e)=1, ~~\forall e\in\tilde\Lambda_{\sigma'} \text { where }\sigma'\in Sing(X), \sigma'\neq \sigma.$$
Hence the reparametrizing cocycle $h_t^s:\Lambda\setminus Sing(X)\to \mathbb{R}$ can be extend to be a continuous cocycle $\tilde h_t^s:\tilde\Lambda \to \mathbb{R}$ with the following properties:
\begin{enumerate}
\item $\tilde h_t^s(e)=h_t^s(x)$, if $\rho(e)=x\in\Lambda\setminus Sing(X)$;
\item $\tilde h_t^s(e)=\|\Phi_t(e)\|$, if $\rho(e)\in S_1$;
\item $\tilde h_t^s(e)=\|\Phi_t(e)\|^{-1}$, if $\rho(e)\in S\setminus S_1$.
\end{enumerate}

To prove  $h_t^s\cdot \psi_t|_{N^{cs}}$ is contracting, we just need to prove that  $\tilde h_t^s\cdot \tilde\psi_t|_{N^{cs}}$ is contracting. We will achieve this by contradiction.

From Proposition \ref{Pro:dominatedsplitting2}, Lemma \ref{Lem:uniformcontracting} and Lemma \ref{Lem:contractingatsingularity}, we can choose constants $l, \eta$ with the following properties:
\begin{itemize}
\item[(A1)] there exist a neighborhood $\mathcal{U}$ of $X$ and a neighborhood $U$ of $\Lambda$ and a constant $N_0>0$ such that for any $Y\in \mathcal{U}$ and any periodic orbit $P$ of $Y$ in $U$ with $\pi(P)>N_0$, one has
\begin{equation}\label{Equ:3}\prod_{i=0}^{[\pi(P)/l-1]}\|\psi^Y_l|_{N^{cs}(\varphi^Y_{il}(p))}\|< e^{-\eta [\pi(P)/l]},\end{equation}
for all $p\in P$;
\item[(A2)] for any periodic point $p\in\Lambda$ of $X$, one has
\begin{equation}\label{Equ:5}\limsup_{n\to+\infty}\frac{1}{n}\sum_{i=0}^{n-1}\log \|\psi^X_{l}|_{N^{cs}(\varphi^X_{il}(p))}\|<-\eta;\end{equation}

\item[(A3)] for any singularity $\sigma\in\Lambda$ with ${\rm Ind}(\sigma)=\alpha$, one has
$$\tilde h_l^s(e)\cdot \|\tilde\psi_l|_{N^{cs}(e)}\|=\|\Phi_l(e)\|\cdot \|\tilde\psi_l|_{N^{cs}(e)}\|< {\rm e}^{-\eta l}, ~~~~\forall e\in\tilde\Lambda_\sigma;$$
for any singularity $\sigma\in\Lambda$ with ${\rm Ind}(\sigma)>\alpha$, one has
$$\tilde h_l^s(e)\cdot \|\tilde\psi_l|_{N^{cs}(e)}\|=\|\Phi_l(e)\|^{-1}\cdot \|\tilde\psi_l|_{N^{cs}(e)}\|< {\rm e}^{-\eta l}, ~~~~\forall e\in\tilde\Lambda_\sigma.$$

\end{itemize}
Assume on the contrary that $\tilde h_t^s\cdot\tilde\psi_t|_{N^{cs}}$ is not contracting, then there exists $e\in\tilde\Lambda$ such that
 $$\prod_{i=0}^{n-1}\tilde h_l^s(\Phi_{il}^\#(e))\cdot\|\tilde\psi_l|_{N^{cs}(\Phi_{il}^\#(e))}\|\geq 1, ~~~\forall n=1,2,\cdots,$$
thus by the continuity of $\tilde h_l^s\cdot\|\tilde\psi_t|_{N^{cs}}\|$, we can find a $\Phi_l^\#$-ergodic measure $\mu$ supported on $\tilde\Lambda$ such that
$$\int_{\tilde\Lambda}(\log( \tilde h_l^s(e))+\log\|\tilde\psi_l|_{N^{cs}(e)}\|)d\mu\geq 0.$$

From the condition $(A3)$ of $l,\eta$, we know that $\mu$ can not be supported on $\cup_{\sigma\in S}\tilde\Lambda_\sigma$. Then we have
$$\int_{\Lambda\setminus Sing(X)}(\log( h_l^s(x))+\log\|\psi_l|_{N^{cs}(x)}\|)d\rho_*\mu\geq 0,$$
where $\rho_*\mu$ denoted the $\varphi_l$-ergodic measure on $\Lambda$ given by $\rho_*\mu(\cdot)=\mu(\rho^{-1}(\cdot))$.
Note that for $\rho_*\mu$-a.e. point $x$, we have
$$\lim_{n\to+\infty}\frac{1}{n}\sum_{i=0}^{n-1}\log( h_l^s(\varphi_{il}(x))=\lim_{n\to+\infty}\frac{1}{n}(\log(h_{nl}^s(x)))$$
$$=\lim_{n\to+\infty}\frac{1}{n}[\sum_{\sigma\in S_1}(\log (k^\sigma(\varphi_{nl}(x)))-\log(k^\sigma(x)))-\sum_{\sigma\in S_2}(\log (k^\sigma(\varphi_{nl}(x)))-\log(k^\sigma(x)))]=0.$$
then we have
$$\int_{\Lambda\setminus Sing(X)}\log(h_l^s(x))d\rho_*\mu=0$$
by the Birkhoff's ergodic theorem. Hence we have
$$\int_{\Lambda\setminus Sing(X)}\log\|\psi_l|_{N^{cs}(x)}\|d\rho_*\mu\geq 0.$$
Then we can take a well closable point $a$ which is also a generic point of $\rho_*\mu$ by the ergodic closing lemma (Proposition \ref{Pro:ergodicclosinglemma}).  By the Birkhoff's ergodic theorem,
 we have
\begin{equation}\lim_{n\to+\infty}\frac{1}{n}\sum_{i=0}^{n-1}\log\|\psi_l|_{N^{cs}({\varphi_{il}(a)})}\|\geq 0.\end{equation}
From condition $(A2)$ of $l,\eta$, we know that $a$ can not be a periodic point of $X$. By the choice that $a$ is also a well closable point, we can find a sequence of $Y_n\to X$ with a periodic point $p_n$ such that the orbit of $p_n$ (under $Y_n$ respectively) close to the orbit segment of $\varphi_{[0,\pi(p_n)]}(a)$. Since $a$ is not periodic, we have $\pi(p_n)\to \infty$ as $n\to\infty$.

\begin{Claim} There is $N_1>0$, such that for any $n>N_1$, we have
$$\frac{1}{[\pi(p_n)/l]}\sum_{i=0}^{[\pi(p_n)/l-1]}\log\|\psi^{Y_n}_l|_{N^{cs}({\varphi^{Y_n}_{il}(p_n)})}\|>-\eta.$$
\end{Claim}
\begin{proof}[Proof of the Claim 1]
Take $K>0$ be a uniform bound of $\log\|\psi_l^{Y}|_{N^{cs}(x)}\|$ on $x\in\Lambda_Y\setminus Sing(Y)$, that is, for any $Y\in\mathcal{U}$ and any $x\in\Lambda_Y\setminus Sing(Y)$, one has $$|\log\|\psi_l^{Y}|_{N^{cs}(x)}\||<K.$$ Since $\rho_*\mu(Sing(X)\cap\Lambda)=0$, we can take an open neighborhood $V$ of $Sing(X)\cap\Lambda$ such that $\rho_*\mu(\partial V)=0$ and $\rho_*\mu (V)<\frac{\eta}{6K}$. Since $a$ is a generic point of $\rho_*\mu$, we have $$\frac{1}{n}\sum_{i=0}^{n-1}\delta _{\varphi_{il}(a)}\to \rho_*\mu$$
in the weak$*$-topology, where $\delta_x$ denotes the atom measure at $x$. Denote $\frac{1}{n}\sum_{i=0}^{n-1}\delta_{\varphi_{il}(a)}$ by $\mu_n$ for all $n\in\mathbb{N}$. Then we can find $N_0>0$ such that for any $n>N_0$, we have $\mu_n(V)<\frac{\eta}{6K}$. This implies
$$\frac{\sharp\{0\leq i<n: \varphi_{il}(a)\in V\}}{n}<\frac{\eta}{6K}$$
for all $n\geq N_0$. From the estimation $(5)$ we can also take $N_0$ with the property
$$\frac{1}{n}\sum_{i=0}^{n-1}\log\|\psi_l|_{N^{cs}({\varphi_{il}(a)})}\|>-\frac{\eta}{3}$$
for all $n>N_0$.
By the continuity of $\psi_t$ and bundle $N^{cs}$ we can take $N_1>0$ and $\delta>0$ such that for any $n>N_1$ and any $x\in\Lambda\setminus V$ and $y\in\Lambda_{Y_n}\setminus V$, if $d(x,y)<\delta$, then $$|\log\|\psi_l|_{N^{cs}(x)}\|-\log\|\psi_l^Y|_{N^{cs}(y)}\||<\frac{\eta}{3}.$$
Without loss of generality, we can assume that for any $n>N_1$, we have $[\pi(p_n)/l-1]>N_0$ and
$$d(\varphi_{il}(a),\varphi^Y_{il}(p_n))<\delta$$
for all $i=0, 1, \cdots, [\pi(p_n)/l-1]$. For shortness of notations, denote $[\pi(p_n)/l]$ by $k_n$. Then we have
$$\frac{1}{k_n}\sum_{i=0}^{k_n-1}\log\|\psi_l|_{N^{cs}({\varphi^{Y_n}_{il}(p_n)})}\|$$
$$=\frac{1}{k_n}(\sum_{0\leq i<k_n, \varphi_{il}(a)\in V}\log\|\psi_l|_{N^{cs}({\varphi^{Y_n}_{il}(p_n)})}\|+\sum_{0\leq i<k_n, \varphi_{il}(a)\notin V}\log\|\psi_l|_{N^{cs}({\varphi^{Y_n}_{il}(p_n)})}\|)$$
$$\geq \frac{1}{k_n}(\sum_{i=0}^{k_n-1}\log\|\psi_l|_{N^{cs}({\varphi_{il}(a)})}\|-\sum_{0\leq i<k_n, \varphi_{il}(a)\in V}2K-\sum_{0\leq i<k_n, \varphi_{il}(a)\notin V}\frac{\eta}{3})$$
$$>-\frac{\eta}{3}-\frac{\sharp\{0\leq i<n: \varphi_{il}(a)\in V\}}{k_n}\cdot 2K-\frac{\eta}{3}\geq -\eta.$$
for all $n>N_1$. This ends the proof of the claim
\end{proof}

This claim contradicts with the condition $(A1)$ on the periodic orbits of $Y$ close to $X$. This proves that $h_t^s\cdot\psi_t|_{N^{cs}}$ is contracting. Similarly we can prove that $h_t^u\cdot\psi_t|_{N^{cu}}$ is expanding. This ends the proof of Theorem \ref{Thm:multisingularpartiallyhyperbolic}.
\end{proof}

At the ending of this section, we insert a conjecture here. Similar to the definition of hyperbolic or partially hyperbolic diffeomorphisms, we can say a vector field $X$ is \emph{multisingular partially hyperbolic} if the chain-recurrence set of $X$ can be split into finitely many compact invariant sets such that each one admits a multisingular partially hyperbolic splitting, whose center bundle can be split into one-dimensional dominated bundles for the linear Poincar\'e flow. Inspired by \cite{CSY15}, one can even have the following conjecture:

\begin{Conj}\label{Con:singular-partially}
Any vector field can be either accumulated by  ones with a homoclinic tangency, or  accumulated by a multisingular partially hyperbolic vector field.
\end{Conj}

{~Note that from the results of da Luz, hyperbolicity and homoclinic tangencies may not be good dichotomy for higher dimensional singular flows. So in the above conjecture, we put the notion ``multisingular partial hyperbolicity'' and homoclinic tangencies as a dichotomy.}

\section{Partial hyperbolicity}

In this section we will prove Theorem \ref{Thm:partial-hyperbolicity}. We give a name for a special case of multisingular partially hyperbolicity before we give the proof of Theorem \ref{Thm:partial-hyperbolicity}.

\begin{Definition}\label{Def:singular-partial}
Let $\Lambda$ be a compact invariant set. We say that $\Lambda$ is
\emph{singular partially hyperbolic}, if $N_{\Lambda\setminus{
Sing}(X)}$ has a dominated splitting $N^{cs}\oplus N_1\oplus N^{cu}$
w.r.t. the linear Poincar\'e flow, and if one of the following cases
occur:
\begin{enumerate}

\item There are constants $C>0$ and $\lambda>0$ such that for any
regular point $x\in\Lambda$ and any $t>0$, one has
$$\frac{\|\psi_t|_{N^{cs}(x)}\|}{\|\Phi_t|_{<X(x)>}\|}\le C{\rm e}^{-\lambda t},~~~\frac{\|\psi_{-t}|_{N^{cu}(x)}\|}{\|\Phi_{-t}|_{<X(x)>}\|}\le C{\rm e}^{-\lambda t}.$$

\item There are constants $C>0$ and $\lambda>0$ such that for any
regular point $x\in\Lambda$ and any $t>0$, one has
$$\frac{\|\psi_t|_{N^{cs}(x)}\|}{\|\Phi_t|_{<X(x)>}\|}\le C{\rm e}^{-\lambda t},~~~\|\psi_{-t}|_{N^{cu}(x)}\|\|\Phi_{-t}|_{<X(x)>}\|\le C{\rm e}^{-\lambda t}.$$

\item There are constants $C>0$ and $\lambda>0$ such that for any
regular point $x\in\Lambda$ and any $t>0$, one has
$$\|\psi_t|_{N^{cs}(x)}\|\|\Phi_t|_{<X(x)>}\|\le C{\rm e}^{-\lambda t},~~~\frac{\|\psi_{-t}|_{N^{cu}(x)}\|}{\|\Phi_{-t}|_{<X(x)>}\|}\le C{\rm e}^{-\lambda t}.$$

\end{enumerate}

\end{Definition}

\begin{Remark}
$(1)$ Liao \cite{Lia89} defined the rescaled linear Poincar\'e flow $\psi_t^*$, which
is defined for any regular point $x$ and any vector $v\in N_x$ by
the following way:
$$\psi_t^*(v)=\frac{\psi_t(v)}{\|\Phi_t|_{<X(x)>}\|}=\frac{|X(x)|\psi_t(v)}{|X(\varphi_t(x))|}.$$
Thus the condition
$\frac{\|\psi_t|_{N^{cs}(x)}\|}{\|\Phi_t|_{<X(x)>}\|}\le C{\rm
e}^{-\lambda t}$ is equivalent to say that $N^{cs}$ is contracting
w.r.t. $\psi_t^*$.

$(2) $ The notion was got in the spirit of singular hyperbolicity.  It is proved in \cite{WWY17} that if  $\Lambda$ is a transitive set and every singularity in $\Lambda$ is hyperbolic, and $\Lambda$ admits a singular-partially
hyperbolic splitting without neutral sub-bundles, then $\Lambda$ is singular hyperbolic.
\end{Remark}

One can verify the singular partially hyperbolic splitting for a nontrivial isolated transitive set with homogenous index on singularities when we consider a $C^1$-generic vector field away from ones with a homoclinic tangency.

\begin{Theorem}\label{Thm:singularpartiallyhyperbolic}
Let $X\in(\mathcal{X}^1(M)\setminus\overline{HT})\cap \mathcal{R}$ and $\Lambda$ be a nontrivial isolated transitive set of $X$. If all singularities in $\Lambda$ have the same index, then $\Lambda$ admit a singular-partially
hyperbolic splitting $N^{cs}\oplus N_1\oplus \cdots\oplus N_k\oplus
N^{cu}$, where $\dim N_i=1$ for any $1\le i\le k$.
\end{Theorem}

\begin{proof}
In Theorem \ref{Thm:multisingularpartiallyhyperbolic} we have proved that there is a mutisingular partially hyperbolic splitting $N_{\Lambda\setminus Sing(x)}=N^{cs}\oplus N_1\oplus \cdots\oplus N_k\oplus N^{cu}$ with two cocycles $h_t^s, h_t^u$ such that $h_t^s\cdot \psi_t|_{N^{cs}}$ is contracting and $h_t^u\cdot\psi_t|_{N^{cu}}$ is expanding. Under the additional assumption that all singularities in $\Lambda$ have the same index, we will prove that this splitting is a singular partially hyperbolic splitting satisfying Theorem \ref{Thm:singularpartiallyhyperbolic}.

Without loss of generality, we assume that there exists at least one singularity in $\Lambda$. Then we have three cases: (1) $\alpha<{\rm Ind}(\sigma)<\beta+1$ for all $\sigma\in\Lambda\cap Sing(X)$; (2) $\alpha<{\rm Ind}(\sigma)=\beta+1$ for all $\sigma\in\Lambda\cap Sing(X)$; (3) $\alpha={\rm Ind}(\sigma)<\beta+1$ for all $\sigma\in\Lambda\cap Sing(X)$, where $\alpha, \beta$ is the minimal and maximal number of $$\{i: \Lambda \text{ cotains an } i-\text{periodic limit}\}.$$

\begin{Claim}
If $\alpha<{\rm Ind}(\sigma)<\beta+1$ for all $\sigma\in\Lambda\cap Sing(X)$, then Condition 1 of Definition \ref{Def:singular-partial} will be satisfied.
\end{Claim}
\begin{proof}[Proof of the Claim]
Let us recall the constructions of $h_t^s, h_t^u$. We know that
$$h_t^s(x)=(\prod_{\sigma\in S}h_t^\sigma)^{-1}= h_t^u(x)$$
where $h_t^\sigma$ is the pragmatic cocycle associated to $\sigma$ and $S=\Lambda\cap Sing(X)$. Here $h_t^\sigma$ was defined by the following way: we defined a positive continuous function $k^\sigma: \Lambda\setminus Sing(x)\to \mathbb{R}$ such that $k^\sigma(x)=\|X(x)\|$ for $x\in V(\sigma)$ and $k^{\sigma}(x)\equiv c$ for $x\notin V(\sigma)$ where $V(\sigma)$ is an isolated neighborhood of $\sigma$ and $c$ is a positive constant, then set $h_t^\sigma(x)=\frac{k^\sigma(\varphi_t(x))}{k^\sigma(x)}$. Now we define a positive continuous function $k:\Lambda\setminus Sing(x)\to \mathbb{R}$ by letting
$$k(x)=\|X(x)\|, ~~\forall x\in \cup_{\sigma\in S}V(\sigma); ~~~~k(x)=c, ~~\forall x\notin \cup_{\sigma\in S}V(\sigma).$$
By a direct computation we know that
$$h_t^s(x)=(\frac{k(\varphi_t(x))}{k(x)})^{-1}=h_t^u(x), ~~\forall x\in \Lambda\setminus Sing(x), t\in\mathbb{R}.$$
There exist two positive constants $$c_1= \frac{\max\{\|X(x)\|:x\in\Lambda\setminus(\cup_{\sigma\in S}V(\sigma))\}}{c}, c_2= \frac{c}{\min\{\|X(x)\|:x\in\Lambda\setminus(\cup_{\sigma\in S}V(\sigma))\}}$$ such that for any $x\in\Lambda\setminus Sing(X)$, one has
$$c_2\|X(x)\|\leq k(x)\leq c_1\|X(x)\|.$$
So we have
$$c_1^{-1}c_2\|\Phi_t|_{<X(x)>}\|=c_1^{-1}c_2\frac{\|X(\varphi_t(x))\|}{\|X(x)\|}\leq h_t^s(x)~~~~~~~~~~~~~~~~~~~~~~~~~~~~~~~~~~~$$
$$~~~~~~~~~~~~~~~~~~~~~~~~~~~~~~~~=h_t^u(x)\leq c_1c_2^{-1}\frac{\|X(\varphi_t(x)\|}{\|X(x)\|}=c_1c_2^{-1}\|\Phi_t|_{<X(x)>}\|.$$
From $h_t^s(x)\cdot\psi_t|_{N^{cs}}$ and $h_t^u\cdot\psi_t|_{N^{cu}}$ is contracting and expanding respectively we know that there exist $C>1, \lambda>0$ such that for any $x\in\Lambda\setminus Sing(X)$ and any $t>0$,
$$h_t^s(x)\cdot\|\psi_t|_{N^{cs}(x)}\|\leq C{\rm e}^{-\lambda t} \text{ and } h_{-t}^u(x)\cdot\|\psi_{-t}|_{N^{cu}(x)}\|\leq C{\rm e}^{-\lambda t}.$$
Thus we have
$$\frac{\|\psi_t|_{N^{cs}(x)}\|}{\|\Phi_t|_{<X(x)>}\|}\le c_1c_2^{-1}C{\rm e}^{-\lambda t},~~~\frac{\|\psi_{-t}|_{N^{cu}(x)}\|}{\|\Phi_{-t}|_{<X(x)>}\|}\le c_1c_2^{-1}C{\rm e}^{-\lambda t}.$$
This proves that Condition 1 of Definition \ref{Def:singular-partial} is satisfied.
\end{proof}

Similarly, in the case of $\alpha<{\rm Ind}(\sigma)=\beta+1$ for all $\sigma\in\Lambda\cap Sing(X)$ we can take a positive continuous map $k:\Lambda\setminus Sing(X)\to\mathbb{R}$ defined by
$$k(x)=\|X(x)\|, ~~\forall x\in \cup_{\sigma\in S}V(\sigma); ~~~~k(x)=c, ~~\forall x\notin \cup_{\sigma\in S}V(\sigma),$$
such that
$$h_t^s(x)=\frac{k(\varphi_t(x))}{k(x)}, ~~h_t^u(x)=(\frac{k(\varphi_t(x))}{k(x)})^{-1}$$
for all $x\in\Lambda\setminus Sing(X)$ and $t\in\mathbb{R}$, where $V(\sigma)$ are isolated neighborhoods for the hyperbolic singularities $\sigma\in\Lambda$. Then by the fact that $h_t^s(x)\cdot\psi_t|_{N^{cs}}$ is contracting and $h_t^u\cdot\psi_t|_{N^{cu}}$ is expanding we can find $C>1, \lambda>0$ such that
$$\frac{\|\psi_t|_{N^{cs}(x)}\|}{\|\Phi_t|_{<X(x)>}\|}\le C{\rm e}^{-\lambda t},~~~\|\Phi_{-t}|_{<X(x)>}\|\cdot\|\psi_{-t}|_{N^{cu}(x)}\|\le C{\rm e}^{-\lambda t}$$
for all $x\in\Lambda\setminus Sing(X)$ and $t>0$. This proves that Condition 2 in Definition \ref{Def:singular-partial} is satisfied in this case.

Similarly, in the case of $\alpha={\rm Ind}(\sigma)<\beta+1$ for all $\sigma\in\Lambda\cap Sing(X)$, Condition 3 in Definition \ref{Def:singular-partial} will be satisfied. Note that we already have $\dim N_i=1$ for any $1\le i\le k$. This proves that in both cases the splitting $N^{cs}\oplus N_1\oplus \cdots\oplus N_k\oplus
N^{cu}$ is the singular partially hyperbolic splitting of the theorem. This ends the proof of Theorem \ref{Thm:singularpartiallyhyperbolic}.
\end{proof}

Here we insert a criterion on partially hyperbolic splitting whose proof is in \cite{WWY17}. See also \cite[Section 2.4]{GaY18}.

\begin{Lemma}\label{Lem:critiriononpartiallyhyp}
Assume $\Lambda$ is a nontrivial transitive set and contains no nonhyperbolic singularities. If there is a dominated splitting $N_{\Lambda\setminus Sing(X)}=N^{cs}\oplus N^{cu}$ w.r.t. $\psi_t$ with constants $C>1,\lambda>0$ such that
$$\|\Phi_t|_{<X(x)>}\|^{-1}\cdot \|\psi_t|_{N^{cs}(x)}\|\leq C{\rm e}^{-\lambda t}$$
for all $x\in\Lambda\setminus Sing(X)$ and $t>0$, then there is a partially hyperbolic splitting $T_\Lambda M=E^{ss}\oplus F$ w.r.t. $\Phi_t$ with $\dim E^{ss}=\dim N^{cs}$.
\end{Lemma}

We prove the following proposition which will imply Theorem \ref{Thm:partial-hyperbolicity}.

\begin{Proposition}\label{Pro:contractingbundle}

Let $X\in(\mathcal{X}^1(M)\setminus\overline{HT})\cap \mathcal{R}$ and $\Lambda$ be a nontrivial isolated transitive set of $X$. If all singularities in $\Lambda$ have the same index and there is a singularity
$\sigma$ whose saddle value $I(\sigma)\geq0$, then $\Lambda$ admits a partially hyperbolic splitting
$T_\Lambda M=E^{ss}\oplus F$
w.r.t. the tangent flow, where $E^{ss}$ is contracting under $\Phi_t$.
\end{Proposition}
\begin{proof}
By Theorem \ref{Thm:singularpartiallyhyperbolic} we know that under the assumptions, there exists a singular partially hyperbolic splitting $N^{cs}\oplus N_1\oplus\cdots N_k\oplus N^{cu}$. Note here we have $I(\sigma)\geq 0$ for some singularities $\sigma\in\Lambda$. By Lemma \ref{Lem:periodiclimitfromsingularity} we know that either $\alpha<{\rm Ind}(\sigma)<\beta+1$ for all $\sigma\in\Lambda\cap Sing(X)$ or $\alpha<{\rm Ind}(\sigma)=\beta+1$ for all $\sigma\in\Lambda\cap Sing(X)$. In both cases we have
$$\frac{\|\psi_t|_{N^{cs}(x)}\|}{\|\Phi_t|_{<X(x)>}\|}\le C{\rm e}^{-\lambda t},$$
for any $x\in\Lambda\setminus Sing(X)$ and $t>0$ by the proof of Theorem~\ref{Thm:singularpartiallyhyperbolic}. By Lemma \ref{Lem:critiriononpartiallyhyp} we know that $\Lambda$ admits a partially hyperbolic splitting $T_\Lambda M=E^{ss}\oplus F$ w.r.t. $\Phi_t$ such that $\dim E^{ss}=\dim N^{cs}$. This completes the proof.
\end{proof}

\begin{proof}[The proof of Theorem~\ref{Thm:partial-hyperbolicity}]
Let $X\in\mathcal{R}\cap(\mathcal{X}^{1}(M)\setminus \overline{HT})$ and $\Lambda$ be a robustly transitive set of $X$. By Theorem \ref{Thm:uniqueindex}, we know that all singularities in $\Lambda$ are hyperbolic and have the same index. If there is a singularity $\sigma\in\Lambda$ with $I(\sigma)\geq0$ then by Proposition \ref{Pro:contractingbundle} we know that $\Lambda$ is partially hyperbolic with a splitting $E^{ss}\oplus E^{cu}$ where $E^{ss}$ is uniformly contracting. If there is a singularity $\sigma\in\Lambda$ with $I(\sigma)\leq 0$, then by Proposition \ref{Pro:contractingbundle} we know that $\Lambda$ is partially hyperbolic with a splitting $E^{cs}\oplus E^{uu}$ where $E^{uu}$ is uniformly expanding. This proves Theorem \ref{Thm:partial-hyperbolicity}.
\end{proof}

\begin{Acknowledgements}
We would like to thank the referees for their important remarks and careful readings. They have pointed out some references. We also communicated with A. da Luz. Her kind suggestions has helped us to improve this paper a lot.

\end{Acknowledgements}

\vskip 5pt

\noindent Xiao Wen

\noindent School of Mathematics and System Science

\noindent Beihang University, Beijing, 100191, P.R. China

\noindent wenxiao@buaa.edu.cn

\vskip 5pt

\noindent Dawei Yang

\noindent School of Mathematical Sciences, Center for Dynamical Systems and Differential Equations

\noindent Soochow University, Suzhou, 215006, P.R. China

\noindent yangdw1981@gmail.com, yangdw@suda.edu.cn

\end{document}